\documentclass{amsart}

\usepackage{amsmath,amssymb}
\usepackage{amsthm}
\usepackage{indentfirst}
\usepackage{mathrsfs}
\usepackage{graphicx}
\usepackage{hyperref}
\usepackage{xcolor}
\usepackage{fourier}
\usepackage[margin=1.5in]{geometry}

\theoremstyle{plain}
\newtheorem{theorem}{Theorem}

\newtheorem{prop}[theorem]{Proposition}

\theoremstyle{definition}

\theoremstyle{remark}
\newtheorem*{remark}{Remark}

\DeclareMathOperator{\Id}{Id}

\newcommand{\ud}{\,\mathrm{d}}
\newcommand{\rd}{\mathrm{d}}

\newcommand{\Or}{\mathcal{O}}

\newcommand{\CC}{{\mathbb C}}

\newcommand{\RR}{{\mathbb R}}

\newcommand{\p}{\partial}

\newcommand{\wh}[1]{\widehat{#1}}

\IfFileExists{mathabx.sty}%
  {\DeclareFontFamily{U}{mathx}{\hyphenchar\font45}%
   \DeclareFontShape{U}{mathx}{m}{n}{<->mathx10}{}%
   \DeclareSymbolFont{mathx}{U}{mathx}{m}{n}%
   \DeclareFontSubstitution{U}{mathx}{m}{n}%
   \DeclareMathAccent{\widebar}{0}{mathx}{"73}%
}{%
  \PackageWarning{mathabx}{%
    Package mathabx not available, therefore\MessageBreak substituting
    widebar with overline\MessageBreak }%
  \newcommand{\widebar}[1]{\overline{#1}}%
}
\newcommand{\wb}[1]{\widebar{#1}}

\newcommand{\veps}{\varepsilon}
\newcommand{\eps}{\epsilon}

\newcommand{\abs}[1]{\lvert#1\rvert}
\newcommand{\norm}[1]{\lVert#1\rVert}

\title[Strang Splitting for QLS]{Strang Splitting Methods for a
  quasilinear Schr\"odinger equation - Convergence, Instability and
  Dynamics}

\author[J. Lu]{Jianfeng Lu}
\address{Departments of Mathematics, Physics, and Chemistry, Duke University, Box 90320, Durham, NC 27708, USA}
\email{jianfeng@math.duke.edu}

\author[J.L. Marzuola]{Jeremy L. Marzuola}
\address{Department of Mathematics, UNC-Chapel Hill \\ CB\#3250
  Phillips Hall \\ Chapel Hill, NC 27599, USA}
\email{marzuola@math.unc.edu}

\date{\today}

\begin{document}

\begin{abstract}
  We study the Strang splitting scheme for quasilinear Schr\"odinger
  equations. We establish the convergence of the scheme for solutions
  with small initial data. We analyze the linear instability of the
  numerical scheme, which explains the numerical blow-up of large data
  solutions and connects to the analytical breakdown of regularity of
  solutions to quasilinear Schr\"odinger equations. Numerical tests
  are performed for a modified version of the superfluid thin film
  equation.
\end{abstract}



\thanks{The first author was supported in part by the Alfred P. Sloan
  Foundation and by the National Science Foundation under grant number
  DMS-1312659.  He would like to thank Weizhu Bao for helpful
  discussions.  The second author was supported as a guest lecturer at
  Karlsruhe Institute of Technology in Summer 2013 as well as by NSF
  grant DMS-1312874, and wishes to especially thank his collaborators
  Jason Metcalfe and Daniel Tataru for introducing him to quasilinear
  Schr\"odinger theory.  We also wish to thank the anonymous reviewers
  for many helpful comments that help improve the exposition in the
  manuscript.  We thank Ludwig Gauckler also for pointing out an error in the convergence proof in an
  earlier version of the draft.}

\maketitle

\section{Introduction}
\label{sec:Intro}

Consider a general quasilinear Schr\"odinger equation
\begin{equation}\label{eq:qls}
  i u_t = - \Delta u + u f(\abs{u}^2) + u g'(\abs{u}^2) \Delta g(\abs{u}^2),
\end{equation}
for $f, g: \RR \to \RR$.  Such equations can be written as
\begin{equation}
\label{eqn:quasiquad}
\left\{ \begin{array}{l}
i u_t + a^{jk} (u ) \p_j \p_ku = 
F(u,\nabla u) , \quad u:
\RR \times \RR^d \to \CC^m \\[3mm] 
u(0,x) = u_0 (x)
\end{array} \right. 
\end{equation}
with small initial data in a space with relatively low Sobolev
regularity.  Note, quadratic quasilinear interactions can also be handled, but with some extra decay assumptions. Here  
\[
a : \CC^m \times (\CC^m)^d \to
\RR^{d \times d}, \qquad 
F: \CC^m \times (\CC^m)^d \to \CC^m
\]
are smooth functions which we will assume satisfy
\[
a(y,z) = I_d+O(|y|^2+|z|^2), \qquad 
F(y,z)= O(|y|^3+|z|^3) \text{ near } (y,z) = (0,0).  
\]

Quasilinear equations of this form have arisen in several models.  See \cite{Poppenberg} for a thorough list, but we mention here works related to the superfluid thin-film equation \cite{Kurihara}, modeling ultrashort pulse lasers \cite{dBHS,dBHNS}, and time dependent density functional theory \cite{CMR:2004}. The model we will consider here numerically equates to setting $g(s) = f(s) = s$, and hence
\begin{equation}
\label{eqn:model}
  i u_t = - \Delta u + \abs{u}^2 u + u \Delta(\abs{u}^2).
\end{equation}
This is a pseudo-attractive version of the superfluid thin-film equation, which is given by
\begin{equation}
\label{eqn:thinfilm}
  i u_t = - \Delta u + \abs{u}^2 u - u \Delta(\abs{u}^2)
\end{equation}  
and can be seen as a leading order contribution to the ultrashort pulse laser models from \cite{dBHS,dBHNS}.
Existence of solutions to quasilinear equations have been studied analytically in several cases, see \cite{dBHS,dBHNS,KPV,KPRV1,KPRV2,MMT3,MMT4,Poppenberg} and many others. The reason we choose to study \eqref{eqn:model} is that, while similar to \eqref{eqn:thinfilm} in that it is guaranteed to have small data local well-posedness from \cite{MMT4} and hence can be used to verify our numerical convergence results for general quasilinear models, the dynamics of \eqref{eqn:model} can lead to a breakdown of regularity due to a non-positive definite conserved energy.  The model \eqref{eqn:thinfilm} on the other hand has a positive energy quantity and, as a result, much more stable dynamics. 

The nonlinear flow will allow interesting singularities to form in the
evolution for large enough initial data.  In particular, we observe
blow-up at a particular amplitude threshold, but these singularities
are representative of a breakdown of regularity in the higher
derivatives and hence not the standard self-similar style blow-up from
the semilinear Schr\"odinger equation.  Such a threshold was observed
as an obstruction to local well-posedness using Nash-Moser type
arguments in \cite{LPT}.  We show analytically that this mechanism for
instability is inherited by the Strang splitting scheme through a
rigorous convergence result and analysis of a finite frequency
approximation.  Moreover, we observe numerically that this threshold
for ill-posedness arises in several different types of initial
configuration and is rather robust.  However, we note that this
threshold is not the numerically observed sharp threshold for
long-time well-posedness, as indeed the dynamics are able to drive
nearby solutions to this critical amplitude.  These features of
\eqref{eqn:model} will be explored in Section \ref{sec:Num}.
 
Let us consider the nonlinear part of the equation
\begin{equation}\label{eq:nonlinear}
  i v_t = v f(\abs{v}^2) + v g'(\abs{v}^2) \Delta g(\abs{v}^2).
\end{equation}
Taking the complex conjugate, we have 
\begin{equation*}
  - i \wb{v}_t = \wb{v} f(\abs{v}^2) 
  + \wb{v} g'(\abs{v}^2) \Delta g(\abs{v}^2). 
\end{equation*}
We calculate 
\begin{equation}
\label{eqn:nonlinampcon}
  \begin{aligned}
    i \partial_t \abs{v}^2 & = i \wb{v} \partial_t v + i v \partial_t
    \wb{v} \\
    & = \abs{v}^2 f(\abs{v}^2) + \abs{v}^2 g'(\abs{v}^2) \Delta
    g(\abs{v}^2) \\
    & \qquad - \abs{v}^2 f(\abs{v}^2) - \abs{v}^2 g'(\abs{v}^2) \Delta
    g(\abs{v}^2) \\
    & = 0,
  \end{aligned}
\end{equation}
and hence under the evolution \eqref{eq:nonlinear} the amplitude is
conserved. This will be a key property used to develop the numerical
scheme.

To construct a stable numerical scheme, we consider a Strang splitting method for the quasilinear Schr\"odinger equation, which is a composition of the exact flows of the differential equations 
\begin{equation}
  i \partial_t u = - \Delta u 
\end{equation}
and 
\begin{equation}\label{eq:splittingB}
  i \partial_t u = u f(\abs{u}^2) + u g'(\abs{u}^2) \Delta g(\abs{u}^2). 
\end{equation}
More concretely, we approximate $u(t_n)$ with $t_n = n \tau$ for a
step size $\tau > 0$ by $u_n$ via
\begin{align}
  & u_{n+1/2}^- = e^{\frac{i}{2} \tau \Delta} u_n; \notag \\
  & u_{n+1/2}^+ = u_{n+1/2}^- \exp\left(-i \tau
    \bigl(f(\abs{u_{n+1/2}^-}^2) + g'(\abs{u_{n+1/2}^-}^2)
    \Delta g(\abs{u_{n+1/2}^-}^2) \bigr)\right);  \label{eqn:splitting} \\
  & u_{n+1} = e^{\frac{i}{2} \tau \Delta} u_{n+1/2}^+. \notag
\end{align}
We note that the scheme is explicit and symmetric, thanks to the
amplitude preserving property \eqref{eqn:nonlinampcon} of \eqref{eq:nonlinear}. One can use a Fourier pseudo-spectral method for the spatial discretization, and hence the flow $\exp(\tfrac{i}{2} \tau \Delta)$ can be efficiently calculated using the fast Fourier transform (FFT), and the flow \eqref{eq:splittingB} amounts to changing the phase of the solution on 
each mesh point.   {  While the pseudospectral numerical flow turns out to be stable, we note that there is a definite loss of derivatives associated with the nonlinear flow component in the middle step of the continuous Strang Splitting algorithm.  This makes iteration of the approximation a challenge without taking smooth initial conditions.  However, modifying the flow to take into account the pseudo-spectral frequency cut-off will allow us to do frequency cut-off dependent estimates. 
}

\smallskip

Due to the advantage of being structure-preserving, the Strang splitting scheme \cite{Strang:68} and higher order splitting schemes (e.g. \cite{st:1993, Yoshida:90}) have been widely applied to a variety of nonlinear Schr\"odinger equations, mainly semilinear Schr\"odinger equations, modeling monochromatic light in nonlinear optics, Bose-Einstein condensates, as well as envelope solutions for surface wave trains in fluids.  See for example \cite{ABB:13,AppelGross:02, BaoCai:13, BaoJakschMarkowich:03, BaoJinMarkowich:03, BaoMauserStimming:03, BaoShen:05, Chin:07, Faou:09, HardinTappert:73, Jin:12, JMS:11, McLachlanQuispel:02, PathriaMorris:87, PerezGarciaLiu:03,ShenWang:13, WeidemanHerbst:86}. While we focus on the Strang splitting scheme for quasilinear Schr\"odinger equations, let us also mention that many other time discretization approaches to solve non-linear evolution equations have been developed, including Crank-Nicholson type schemes (see e.g. \cite{Sanz-Serna:84} and also \cite{ADKM} and \cite{HMZ1} for applications in studying numerical blow-ups and nonlinear scattering), Magnus expansion approaches (\cite{Magnus:54} and also the recent review article \cite{BCOR:2009}), exponential time-differencing schemes (see e.g. \cite{CoxMatthews:02, kt:2005}), implicit-explicit methods (see e.g. \cite{AscherRuuthWetton:95}),  the comparison study in \cite{TahaAblowitz:84}, and many others.

The convergence of splitting schemes for semilinear Schr\"odinger
equations was analyzed in \cite{DescombesThalhammer:13, Gauckler:11,
  JahnkeLubich:00, Lubich:08, ShenWang:13, Thalhammer:08}. In the
present work, we extend the previous works to the quasilinear
Schr\"odinger equation. { We will prove the convergence of the
time-splitting method to the original evolution for the superfluid
thin-film equation by first proving convergence to a mollified flow, then using convergence of the mollified flow to full continous quasi-linear problem. The analysis follows the ideas in the seminal
contribution by Lubich in \cite{Lubich:08}, where the main tools are
the calculus of Lie derivatives.  However, we cannot really treat even our mollified problem as fully semi-linear since the mollified equation estimates as a semi-linear problem introduce losses that depend upon the choice of frequency cut-off that we wish to avoid for the sakes of uniform estimates.} We will emphasize on the regularity
of the time flow, for which the behavior of the quasilinear
Schr\"odinger equation is different from the semilinear ones.  This is
a Lie theoretic approach to the continuous time approximation and
Sobolev-based well-posedness results of the second author with
J. Metcalfe and D. Tataru in order to model small initial data
solutions of finite time intervals \cite{MMT3, MMT4}.  The scheme is
symplectic and is stable within a range of parameters, motivated by
the analysis in \cite{MMT4}, where the analysis is done purely in
Sobolev spaces $H^s$ for $s$ sufficiently large.  In addition, the
Strang splitting method converges in the order $\tau^2$ for time step
$\tau$.

Moreover, we are able to extend a linear instability observed in a quasilinear Schr\"odinger equation in \cite{LPT} to the numerical scheme used to approximate it, justifying the accuracy of a numerically observed blow-up.  We study the dynamics of this blow-up solution using both Gaussian and plane-wave configurations of initial data to observe that the threshold for instability is not the sharp global well-posedness threshold for the equation and can indeed be reached through frequency dynamics on lower-amplitude solutions.  This is not the standard blow-up through re-scaling of a nonlinear state, but is a frequency instability of sorts that causes high frequencies to grow exponentially in a method akin to a backwards heat map.

\smallskip

The result is laid out as follows.  We begin with a numerical study of
the modified superfluid film equation in $1D$ \eqref{eqn:model} using
the Strang splitting scheme \eqref{eqn:splitting}.   {  To analyze the
convergence of the scheme, we discuss the mollification argument and prove the convergence of the mollified numerical
scheme to the mollified flow for small data in Section \ref{sec:Lie} by using Lie theoretic
results and necessary multilinear estimates. } To understand the blow-up
behavior observed for large enough data, we analyze the stability and
instability of the scheme in Section~\ref{sec:Stab}.  In Section
\ref{sec:Reg}, we discuss the regularity of the time flow of the
quasilinear Schr\"odinger equation and the time-splitting scheme.  In \cite{LPT} an $L^\infty$ threshold for local well-posedness was observed through use of Fr\'echet derivatives in a Nash-Moser scheme.  We will show this similarly arises in analysis about exact plane-wave solutions on the torus using analysis similar to that of  \cite{WeidemanHerbst:86}.
In order to establish the differentiability of the numerical solution with respect to time to sufficiently high accuracy, we rely on bounds in a much stronger topology in space.  {  Finally, we tie together the numerical scheme and full quasilinear flow by addressing the convergence of mollified quasilinear equations to the full quasilinear flow in Section \ref{sec:mollify}.}

\section{Numerical Results}
\label{sec:Num}

To test the Strang splitting scheme \eqref{eqn:splitting} for
quasilinear Schr\"odinger equations and numerically study the
regularity breakdown, we consider the modified superfluid thin-film
equation (see \eqref{eqn:model} and also \cite{Kurihara,Poppenberg})
given by
\begin{equation}\label{eq:superfluid}
  i u_t + u_{xx} = |u|^2 u + (\abs{u}^2)_{xx} u 
\end{equation}
on the domain $(-\pi, \pi]$ with periodic boundary condition. We have
done similar computations for the ultrashort pulse laser equation as
described in \cite{dBHS,dBHNS}, but no further interesting features of
the numerical analysis arose, so we do not present them here for
clarity of exposition.

\subsection{Symmetric Gaussian initial condition}
Let us consider initial conditions given by
\begin{equation}\label{eq:initgauss}
  u(0, x) = a e^{- x^2 / (2 \sigma^2)},
\end{equation}
on the domain $(-\pi, \pi]$ with periodic boundary condition. Here
$\sigma$ is the width and $a$ is the amplitude of the Gaussian
profile.

We calculate the solution up to time $T = \pi/4$ with parameters $a =
1/5$ and $\sigma = 1/5$ for the initial condition.  A Fourier
pseudo-spectral method with $N = 256$ spatial grid points is used. To
verify the second order accuracy of the time-splitting scheme, we
choose different numbers of time steps and estimate the error by
comparing the numerical solutions to a solution with $N_t = 10^5$. The
results in table~\ref{tab:2ndorder} and figure~\ref{fig:2ndorder}
confirm the second order convergence.
\begin{table}[htp]
  \begin{tabular}{ccc}
    $N_t$ & $L^2$-norm & $H^1$-seminorm \\
    \hline
    $500$ & $1.6973e-06$ & $4.7783e-04$ \\
    $1000$ & $4.2241e-07$ & $1.1878e-04$ \\
    $2000$ & $1.0545e-07$ & $2.9642e-05$ \\
    $4000$ & $2.6323e-08$ & $7.3990e-06$ \\
    $8000$ & $6.5487e-09$ & $1.8407e-06$ 
  \end{tabular}
  \smallskip
  \caption{Numerical error of the time-splitting scheme for initial 
    data \eqref{eq:initgauss} with $a = 1/5$ and $\sigma = 1/5$ at 
    $T = \pi/4$.}   
  \label{tab:2ndorder}
\end{table}

\begin{figure}
  \centering
  \includegraphics[width = 3in]{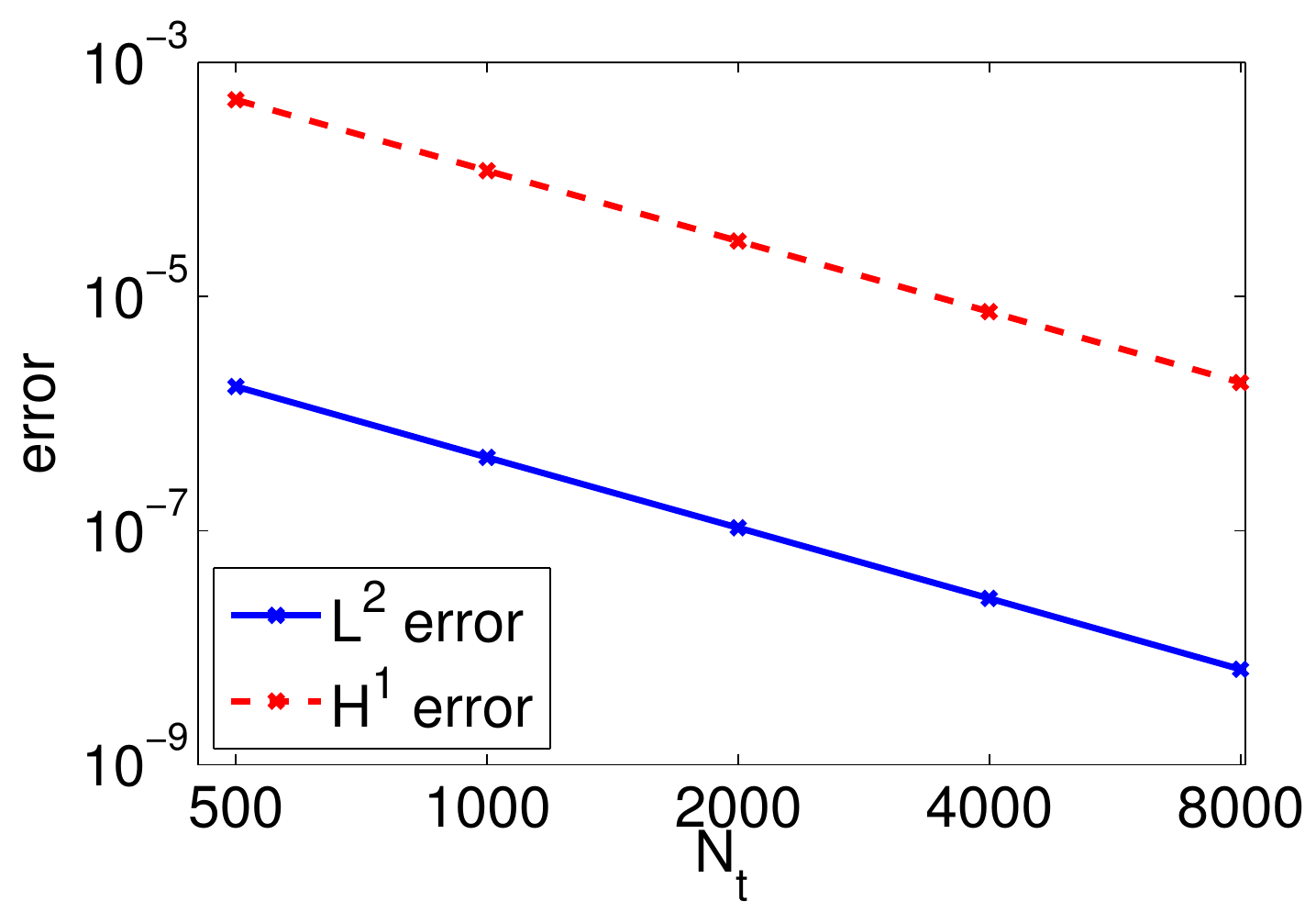}
  \caption{Log-log plot of the numerical error in
    table~\ref{tab:2ndorder} measured in $L^2$ norm and $H^1$
    seminorm.}\label{fig:2ndorder}
\end{figure}

Provided the solution remains in $H^1$, the PDE \eqref{eq:superfluid} conserves mass and energy given by 
\begin{align}
\label{eqn:mass}
  & M(u) = \int \abs{u}^2 \ud x, \\
 \label{eqn:energy}
  & E(u) = \frac{1}{2} \int \abs{u_x}^2 \ud x + \frac{1}{4} \int \abs{u}^4
  \ud x - \frac{1}{4} \int \abs{ (\abs{u}^2)_x}^2 \ud x.
\end{align} 
The numerical scheme conserves the mass conservation law.  While there
is no energy conservation law \cite[table $1$]{ABB:13}, 
the energy is observed to remain numerically conserved with tiny deviation, as
shown in figure~\ref{fig:conserv}.

\begin{figure}[ht]
  \centering
  \includegraphics[width = 3in]{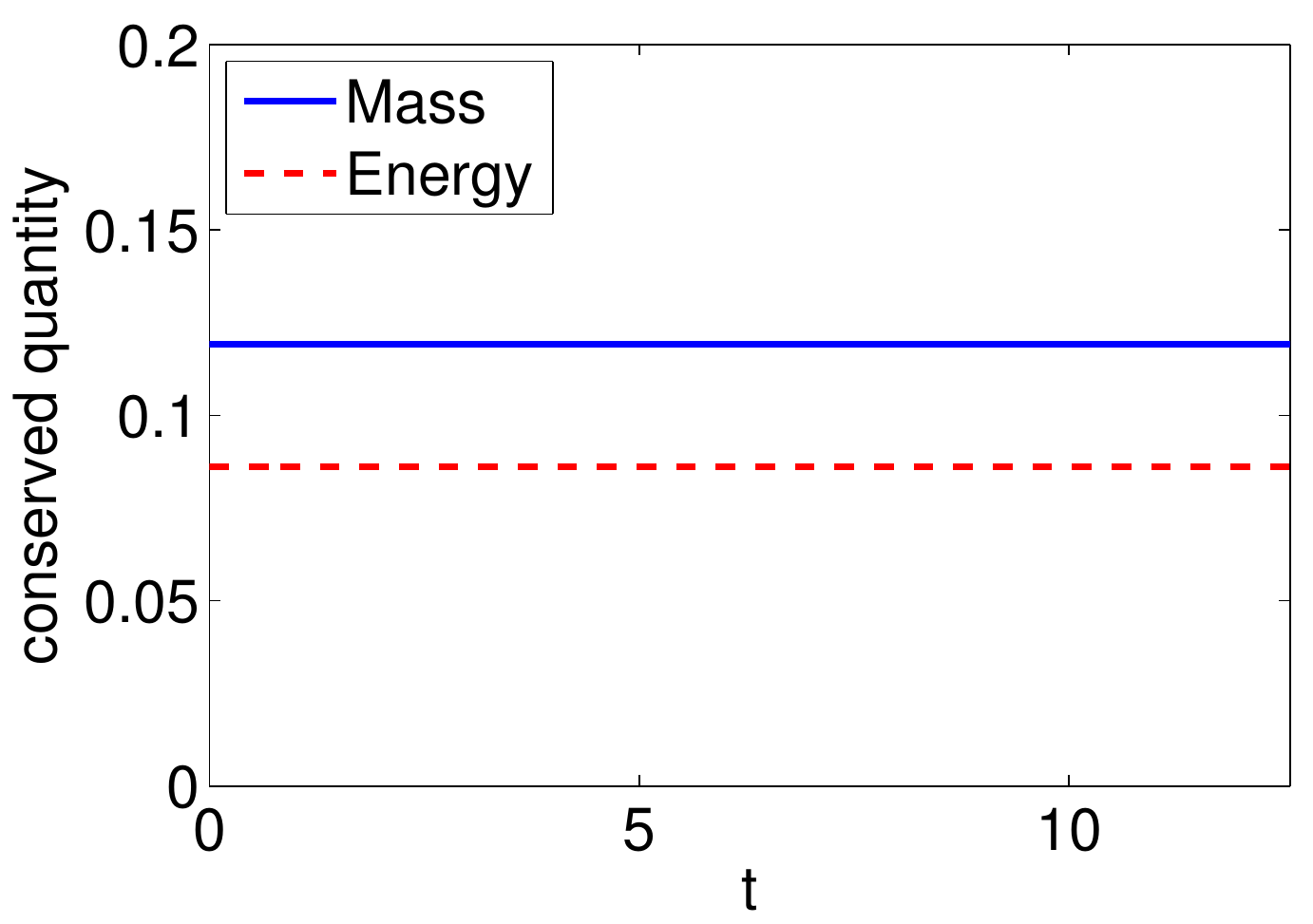}
  \caption{Mass and energy conservation for the numerical solution.
    The initial condition is given by \eqref{eq:initgauss} with $a =
    1/5$ and $\sigma = 1.5$. The numerical solution is calculated up
    to time $T = 4 \pi$ and with $N_t = 32000$ time steps. Energy and
    mass are recorded every $100$ time steps.}\label{fig:conserv}
\end{figure}

Due to the nonlinearity of the equation \eqref{eq:superfluid}, the
problem becomes more stiff for initial conditions with larger amplitude
or derivatives. For the family of Gaussian initial data
\eqref{eq:initgauss}, this means to increase $a$ or reduce $\sigma$. We
next consider the example with $a = 0.625$, $\sigma = 1/10$ and $T =
\pi/4$. The problem is considerably more difficult than the previous
choice of parameters. We refine the spatial discretization to $N =
512$ to resolve the oscillatory profile of the solution. The numerical
error can be found in table~\ref{tab:nearblowup}. We still observe
second order accuracy, though in this case, the time step size
cannot be too large, otherwise the numerical scheme becomes unstable.

\begin{table}[htp]
  \begin{tabular}{ccc}
    $N_t$ & $L^2$-norm & $H^1$-seminorm \\
    \hline
    $10000$ & unstable \\
    $20000$ & $3.5573e-04$ & $1.2775e-01$ \\
    $40000$ & $3.0928e-04$ & $5.1125e-02$ \\
    $80000$ & $8.2591e-05$ & $1.5832e-02$ \\
    $160000$ & $1.9494e-05$ & $4.1212e-03$ \\
    $320000$ & $4.4489e-06$ & $9.6345e-04$
  \end{tabular}
  \smallskip
  \caption{Numerical error of the time-splitting scheme for initial 
    data \eqref{eq:initgauss} with $a = 0.625$ and $\sigma = 1/10$ at 
    $T = \pi/4$. The error is estimated by comparing the numerical 
    solution with $N_t = 10^6$.}   
  \label{tab:nearblowup}
\end{table}

We remark that to make the scheme more stable, it is possible to apply
Fourier spectrum truncation to eliminate spurious Fourier components
of the numerical solution, as introduced in~\cite{Krasny:86}. At each time step, we set to zero all Fourier
coefficients with amplitude below a certain threshold $\delta$ times
the maximum amplitude of Fourier coefficients. In practice, for this
example, we find the threshold $\delta = 1e-3$ makes the scheme stable
with $N_t = 2000$ (recall that the solution is not stable for $N_t =
10000$ without Fourier truncation). On the other hand, the filtering
might introduce inconsistency to the numerical results.

\begin{figure}[htb]
  \centering
  \includegraphics[width = 3in]{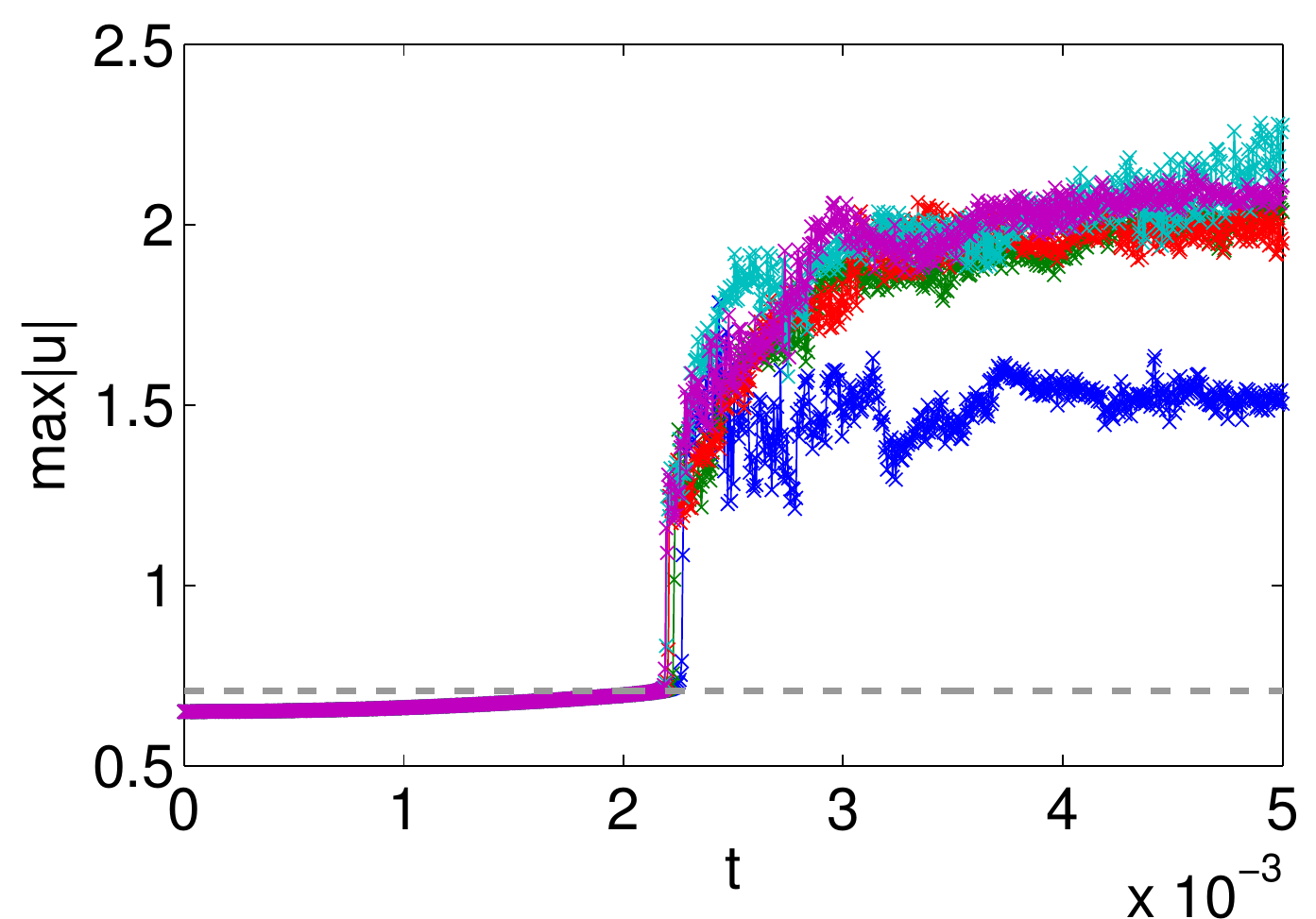} 
  \includegraphics[width = 3in]{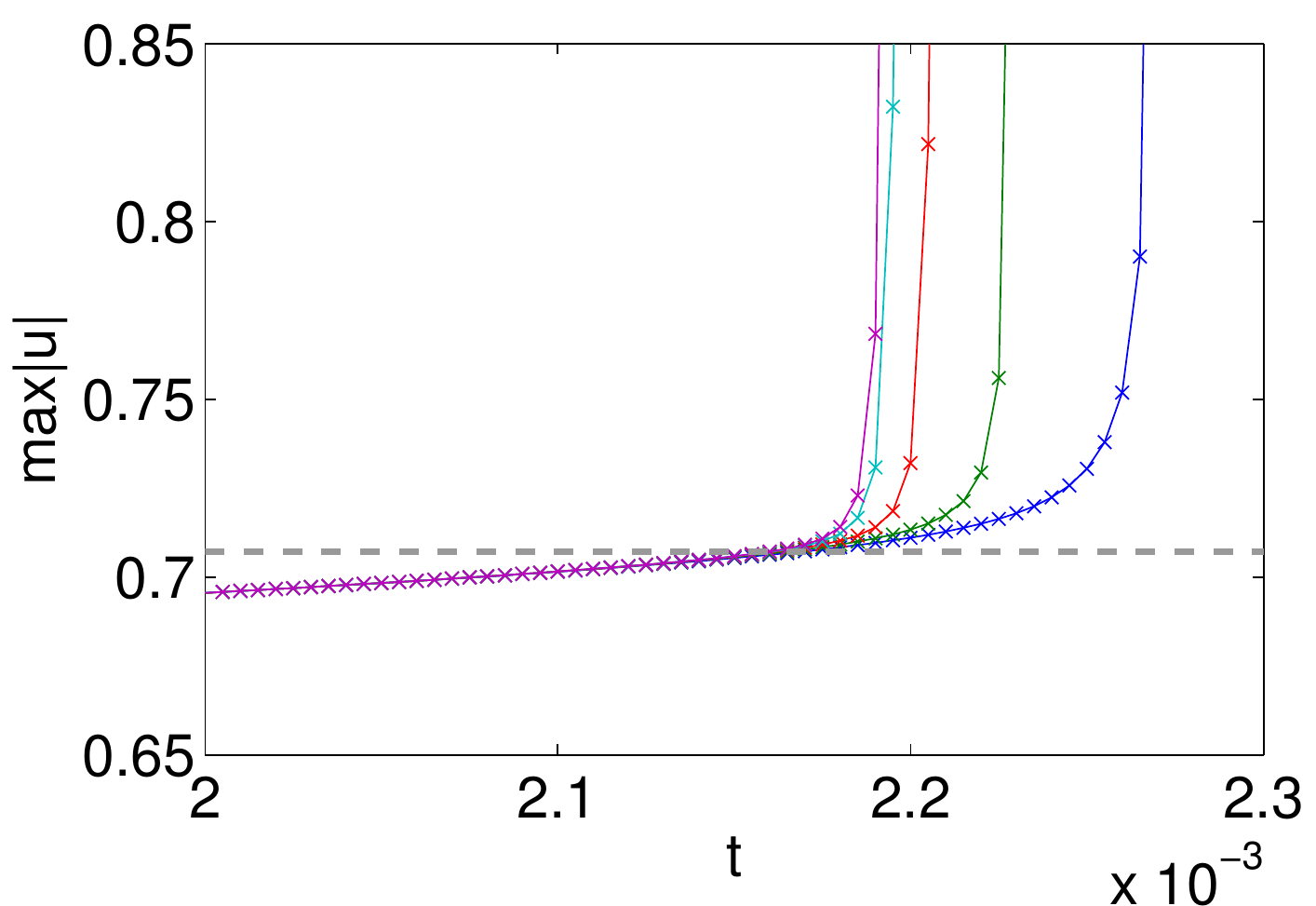}
  \caption{$\max \abs{u(\cdot, t)}$ as a function of time for Gaussian
    initial condition with $a = 0.65$ and $\sigma = 1/10$. Five time
    step sizes are taken, correspond to $N_t = 10000, 20000, 40000,
    80000$, and $160000$ (blue, green, red, cyan, and purple curves)
    for total simulation time $T = \pi$. The bottom panel zooms in the
    region near the numerical blow-up. The dashed horizontal line
    indicates the level $\sqrt{2}/2$.}
  \label{fig:umax_blowup}
\end{figure}

If we further increase the amplitude of the initial condition, the
numerical results indicate a ``blow-up'' behavior for the PDE. We
increase the amplitude to $a = 0.65$ while keeping $\sigma = 1/10$.
The numerical solution is calculated up to $T = 5 \times
10^{-3}$. figure~\ref{fig:umax_blowup} shows $\max \abs{u(\cdot, t)}$
as a function of $t$ for different choices of time step sizes. The
sudden jump and exponential increase of the magnitude of maximum
around $t = 2.18 \times 10^{-3}$ indicates a numerical ``blow-up'' of
the solution. Note that the onset point of the behavior does not
depend on the choice of time step size, indicating that this is not
due to numerical instability of the time integration.  Here we have
chosen $N = 4096$ spatial grid points. The blow-up behavior persists
for further refinement of the spatial discretization.

To investigate more closely the above observed ``blow-up'', we study
the solution around $x = 0$ and the time when the ``blow-up'' occurs.
We plot the absolute value of the solution in
figure~\ref{fig:u_blowup}. The numerical simulation indicates that the
solution develops a ``focusing peak'' at $x = 0$ with amplitude close
to $\sqrt{2}/2$.

\begin{figure}[htp]
  \centering
  \includegraphics[width = 3in]{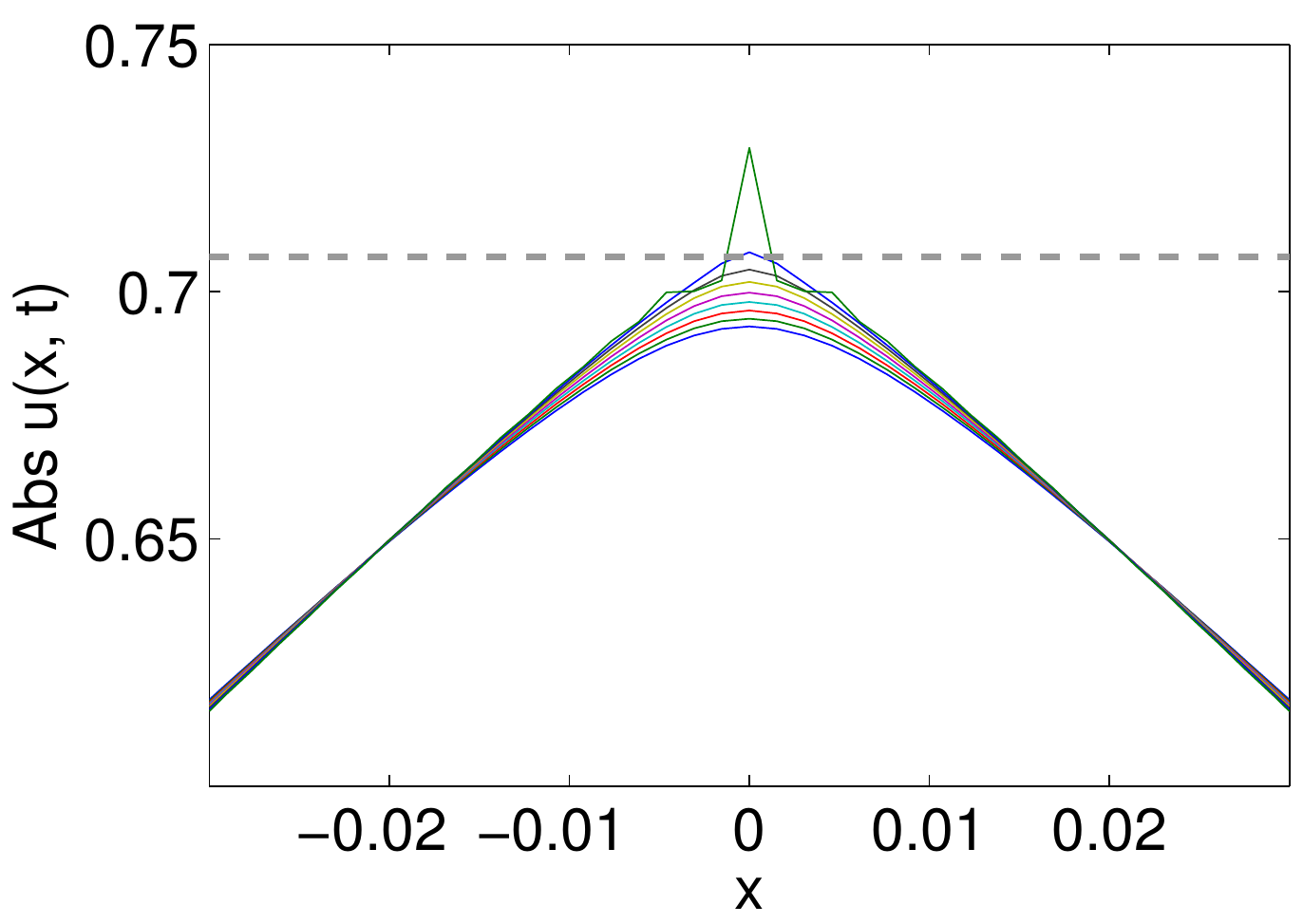}
  \caption{Snapshots of the absolute value of the numerical solutions around
    the numerical blow-up. The solution squeezes as time increases and
    leads to a blow-up. The reference horizontal line is plotted at
    the value $\sqrt{2}/2$.}
  \label{fig:u_blowup}
\end{figure}

The numerical results suggest that the solution to the PDE becomes
unstable for this family of initial conditions when the amplitude
reaches around $\sqrt{2}/2$. To further confirm this, we compare the
results for the initial condition with $a = 0.625$ and $\sigma =
1/10$, the solution stays below the amplitude of $\sqrt{2}/2$ as in
figure~\ref{fig:noblowup}. Numerically, no blow-up is observed for $a
= 0.625$. The instability for large time step size is caused by
pollution in the Fourier spectrum, but not the intrinsic instability
of solutions to the PDE.  

\begin{figure}[htp]
  \centering
  \includegraphics[width = 3in]{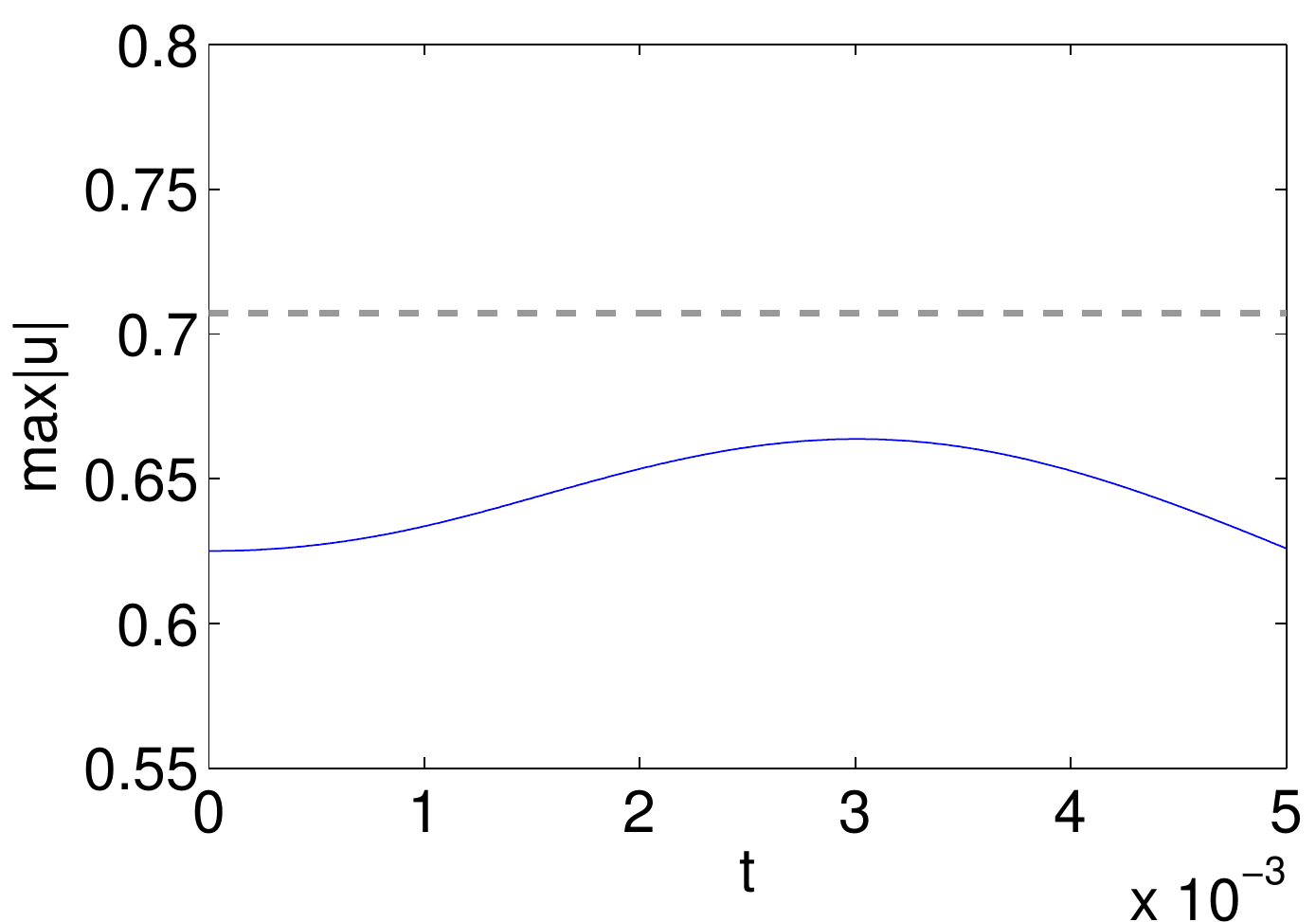}
  \caption{The maximum magnitude of $u$ as a function in time for
    initial condition with $a = 0.625$ and $ \sigma = 1/10$. Compare
    with the quite different behavior in figure~\ref{fig:umax_blowup}.}
  \label{fig:noblowup}
\end{figure}

\subsection{Plane-wave initial conditions}
\label{s:pw}

The only exact solution we are aware of for the superfluid equation
\eqref{eq:superfluid} is the family of wave trains:
\begin{equation}\label{eq:planewave}
  u(x, t) = a \exp i (kx - \omega t).
\end{equation}
This is a solution to \eqref{eq:superfluid} provided that
\begin{equation}
  \omega = k^2 + \abs{a}^2.  
\end{equation}
Since $\abs{u(x, t)} = \abs{a}$ for the solution \eqref{eq:planewave}
at any $x$ and $t$, the splitting error of the Strang splitting scheme
vanishes, as the potential commutes with the $\Delta$ operator.

We study the instability by adjusting the amplitude $a$ of the initial
data $u(x, 0) = a \exp (i kx)$ of the solution \eqref{eq:planewave}.
figure~\ref{fig:planewave} shows the simulation results for two
solutions with initial conditions given by \eqref{eq:planewave} with
$a = \sqrt{2}/2 - 10^{-8}$ and $a = \sqrt{2}/2 + 10^{-8}$,
respectively. Even though the amplitudes of the two solutions only
differ by $2\times 10^{-8}$, the behavior of the numerical solutions
are completely different. While the numerical solution for the former
is stable and accurate, the local truncation error kicks off
instability in the latter case. This indicates again that $\sqrt{2}/2$
is the threshold of instability.

\begin{figure}[htp]
  \includegraphics[width = 3in]{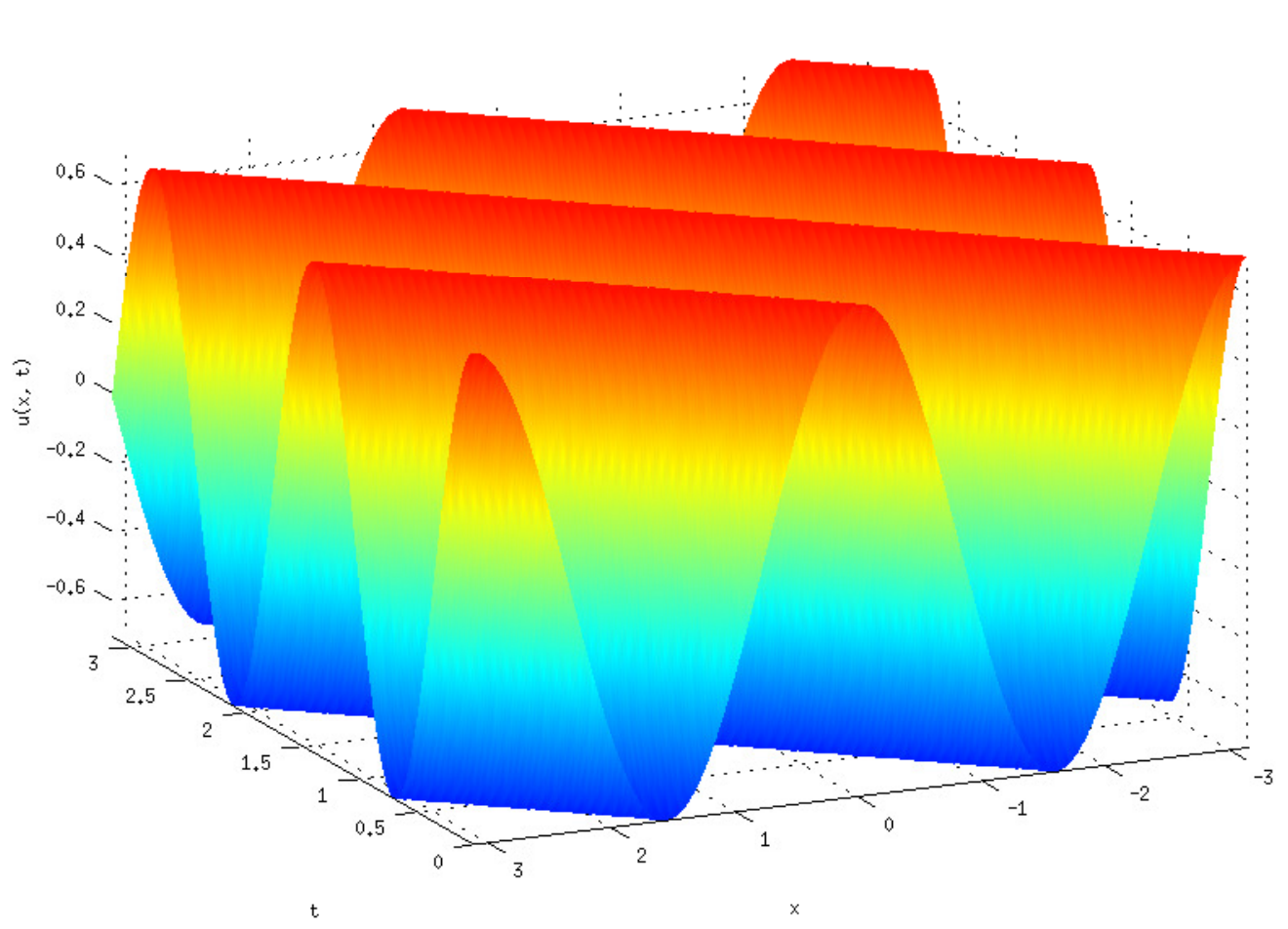}
  \includegraphics[width = 3in]{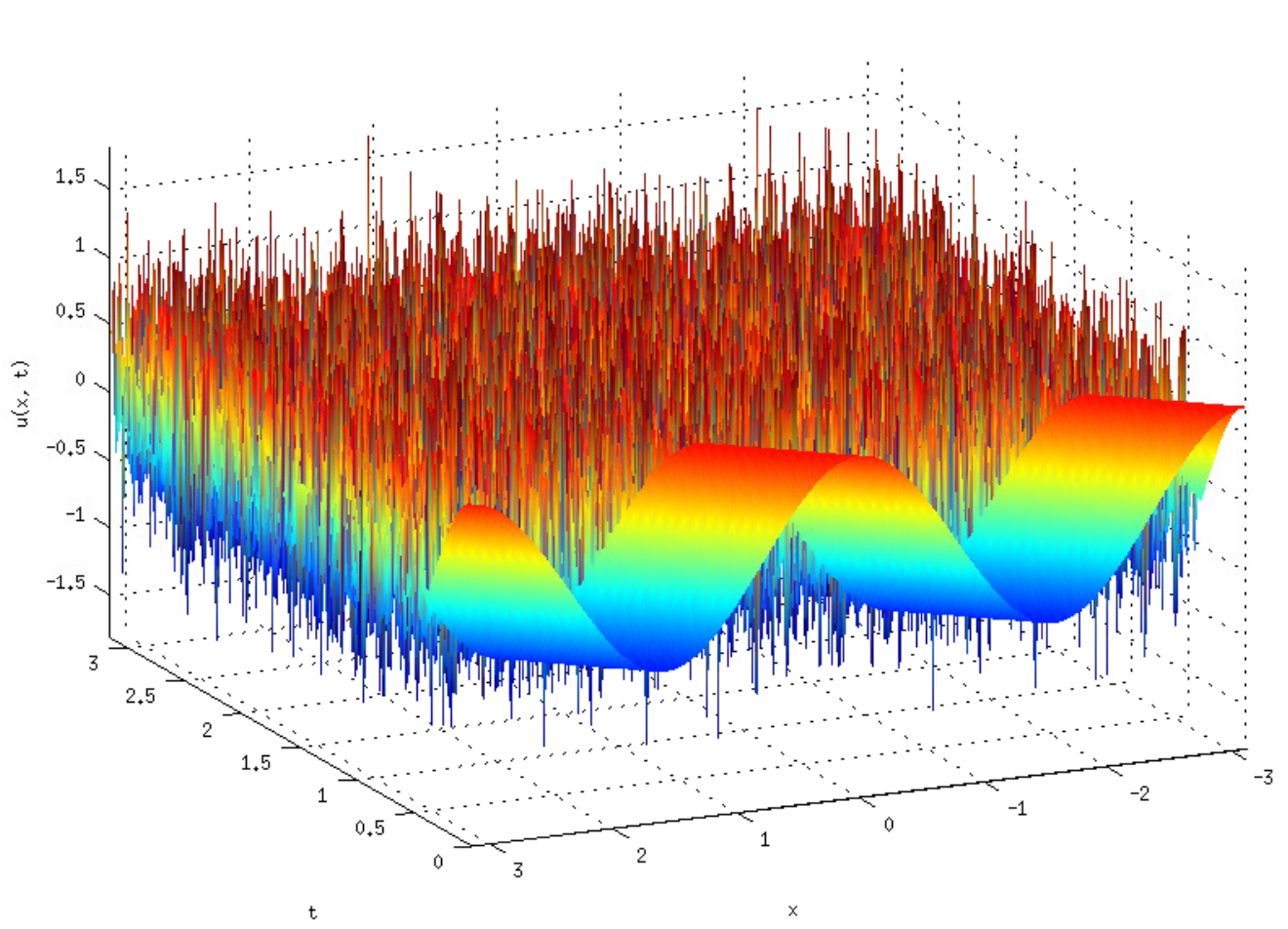}
  \caption{Numerical solution for \eqref{eq:planewave} with $a =
    \sqrt{2}/2 - 10^{-8}$ and $a = \sqrt{2}/2 + 10^{-8}$,
    respectively. The solution to the PDE is unstable when the
    amplitude is larger than $\sqrt{2}/2$.}
  \label{fig:planewave}
\end{figure}

We also study multiple Fourier mode solutions to observe if non-local interactions can vary the blow-up profile.  Hence, given a pseudospectral discretization scheme keeping the first $N$ Fourier modes, we take initial data of the form
\begin{equation}\label{eq:multiplanewave}
  u(x, 0) = a \sum_{j = 1} \exp i k_j x ,
\end{equation}
for $0 \leq k_1 \leq \dots \leq k_j \ll N$. The blow-up behavior of these  solutions become more complicated; in particular, oscillations begin to factor around the blow-up after an initial exponential growth of the maximum amplitude (see figure~\ref{fig:multmodes}). However, it seems that generically $\sqrt{2}/2$ is still a threshold for blow-up.
\begin{figure}[ht!]
  \centering
  \includegraphics[width = 2in]{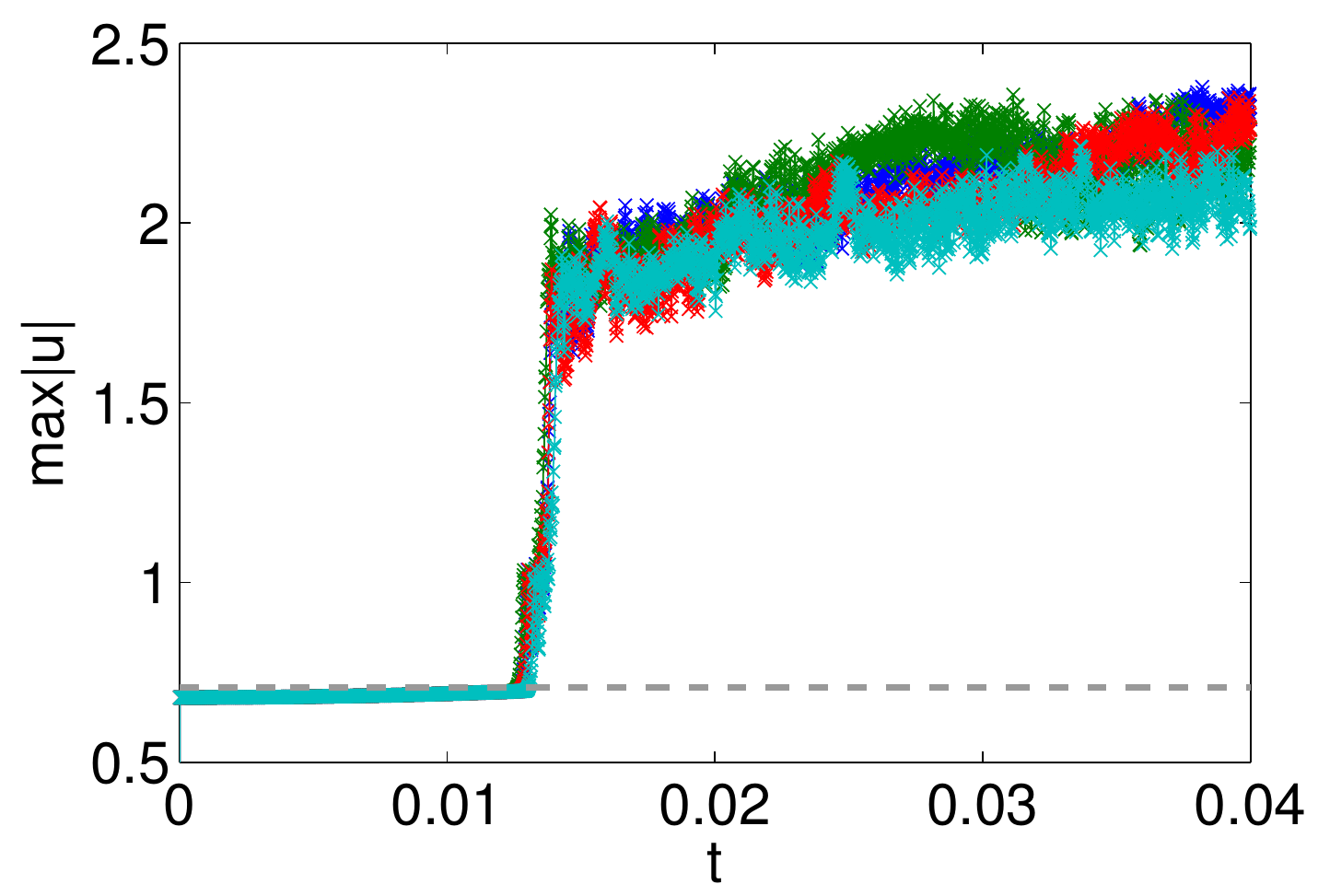}  
  \includegraphics[width = 2in]{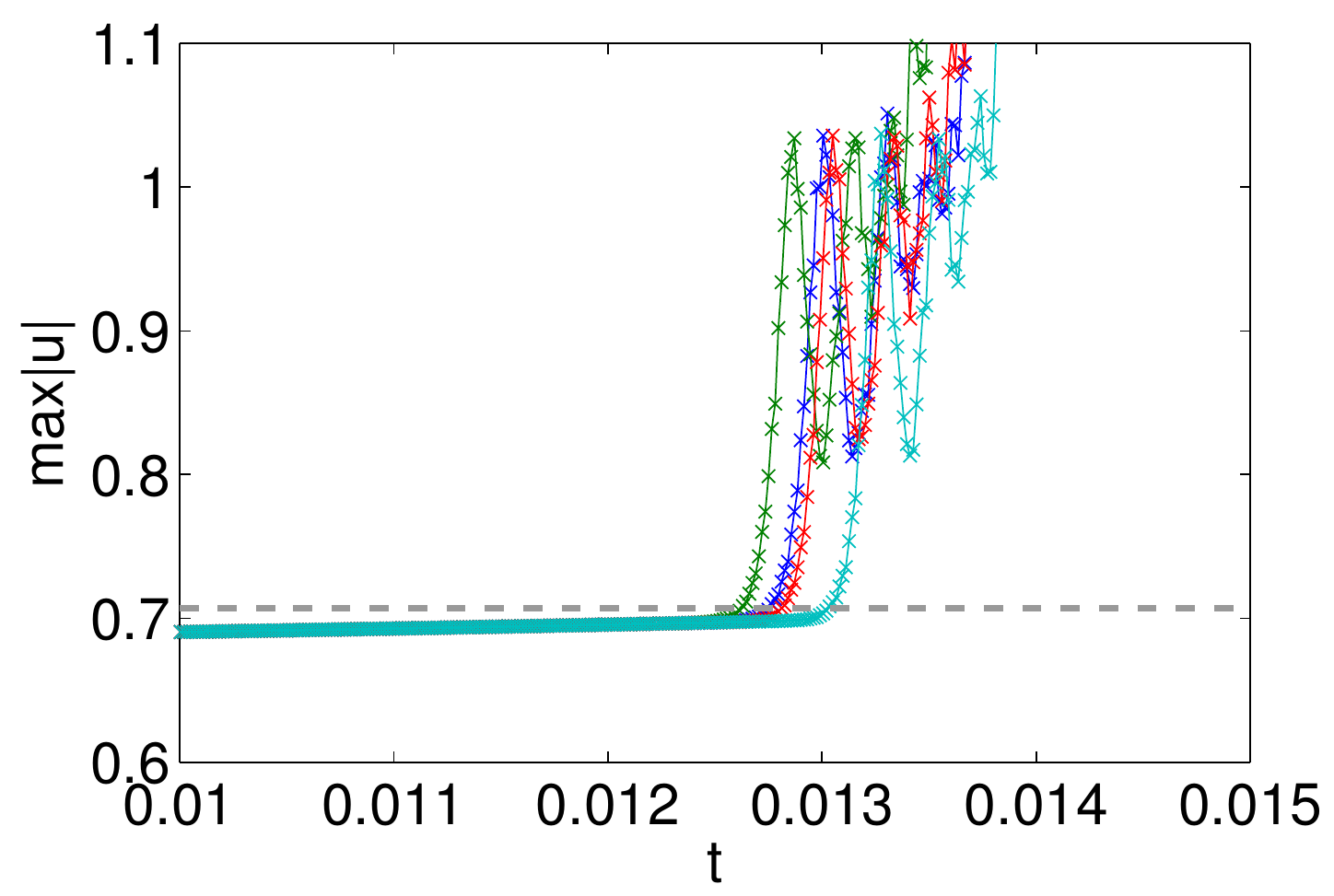}  \\
  \includegraphics[width = 2in]{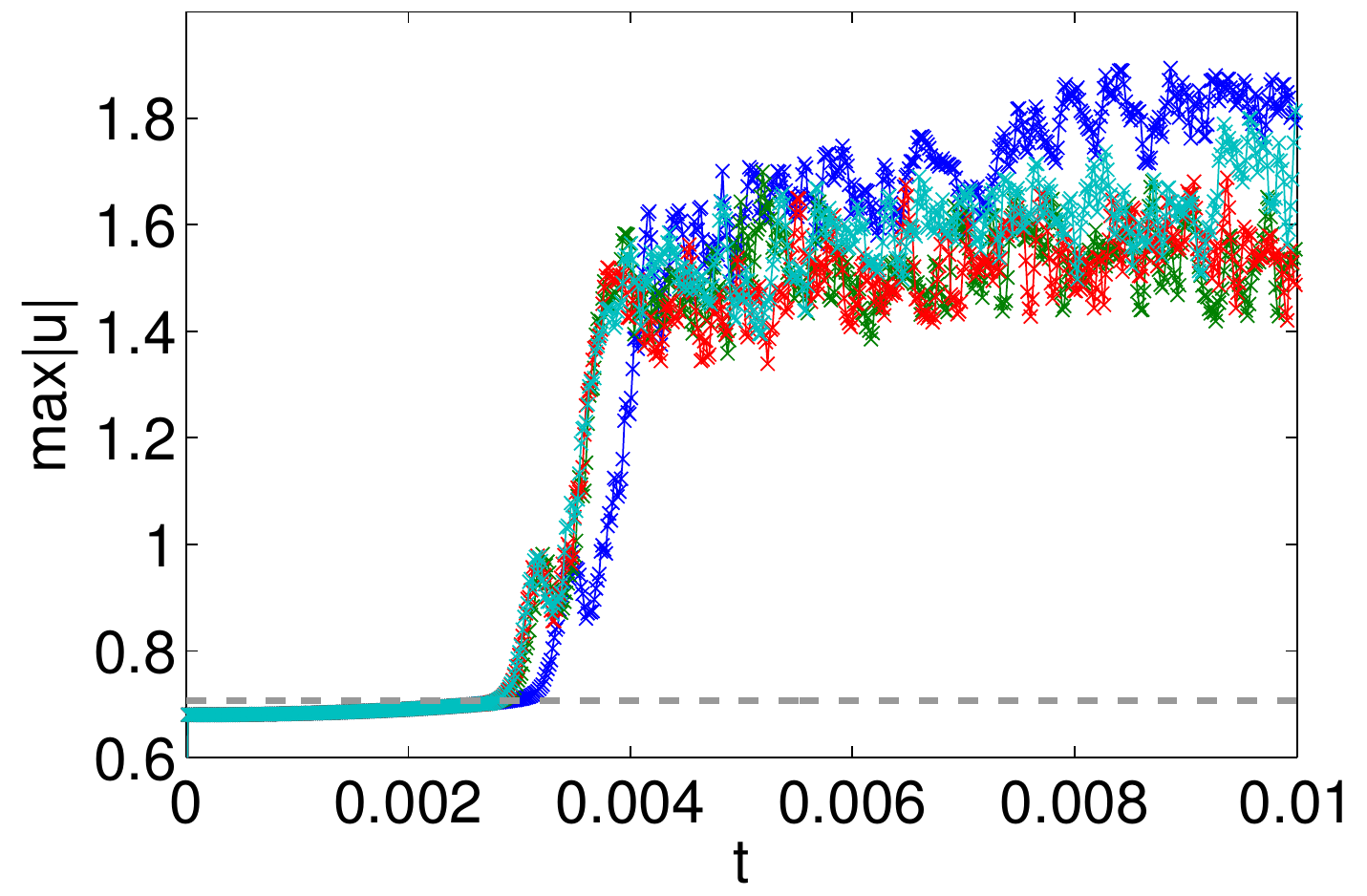}
  \includegraphics[width = 2in]{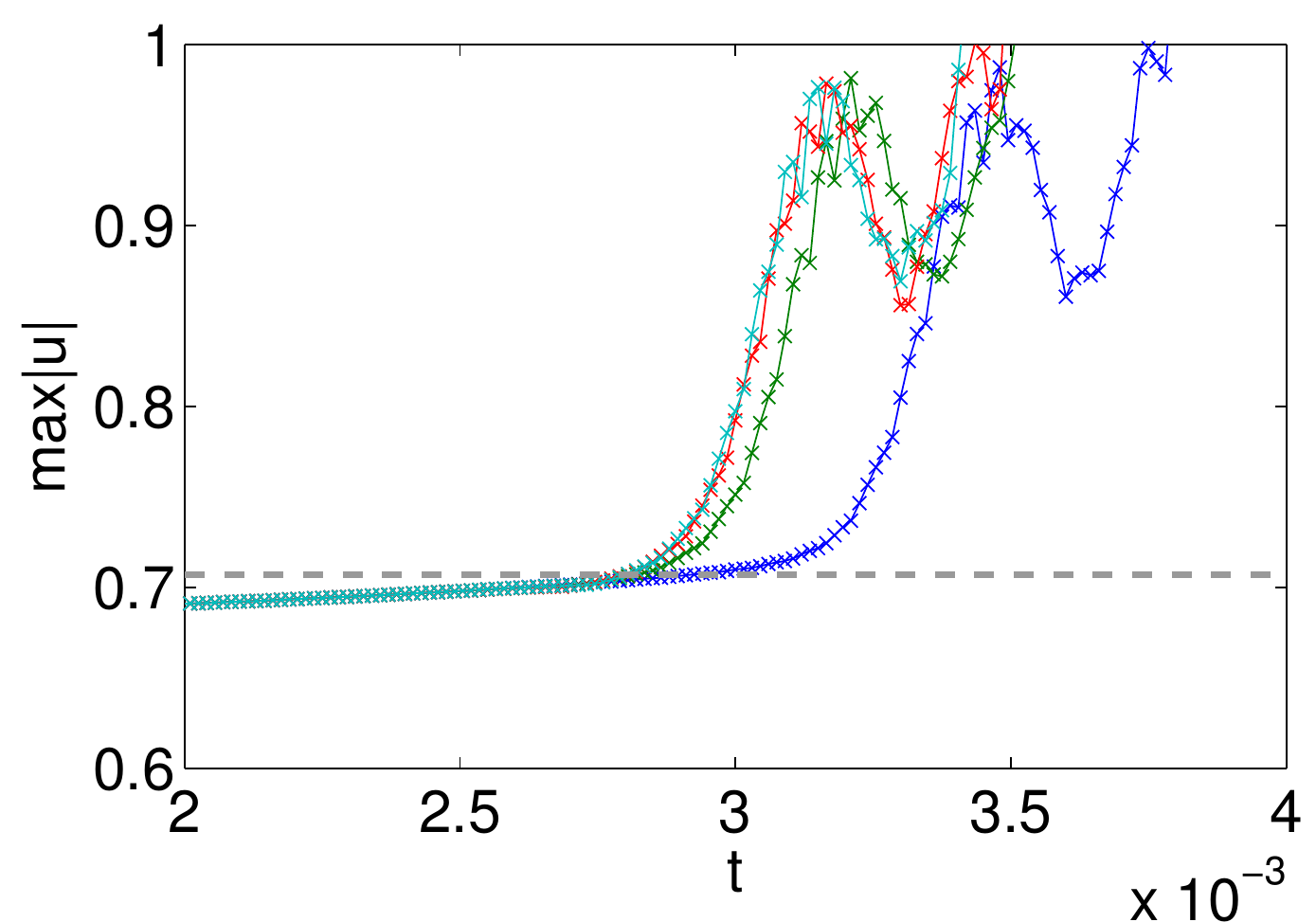}
  \caption{$\max \abs{u(\cdot, t)}$ as a function of time for multiple
    Fourier mode initial condition with $a = 0.65$ and frequencies
    $2,8$ in the $2$-mode setting (top), and $2,8,14,20$ in the
    $4$-mode setting (bottom). Four time step sizes are taken,
    correspond to $N_t = 10000, 20000, 40000$, and $80000$ (blue,
    green, red, and cyan curves) for total simulation time $T =
    0.15$. The right panel zooms in the region near the numerical
    blow-up.  The dashed horizontal line indicates the level
    $\sqrt{2}/2$.}
  \label{fig:multmodes}
\end{figure}

\section{Convergence of the pseudo-spectral time-splitting scheme}
\label{sec:Lie}

{ In this section, we prove the convergence of a modified
  Strang splitting scheme to closely match the pseudospectral scheme
  used numerically for small data initial conditions. We suppose that
  the solution $u(t)$ to the modified superfluid thin film equation
  \eqref{eq:superfluid} in $1d$ is in $H^7$ for $0 \leq t \leq T$, and
  wish to compare to a Strang splitting flow defined such that an
  implicit frequency cut-off occurs at each stage of computation with
  the quasilinear nonlinearity.  We show that the pseudo-spectral
  Strang splitting is well approximated by a mollified superfluid thin
  film equation.  As the analysis is somewhat dissimilar to that
  presented in this section related to the numerical algorithm, later
  in Section \ref{sec:mollify}, we will compare the evolution of the
  mollified superfluid thin film equation to \eqref{eq:superfluid}.

\begin{theorem}
  \label{thm:conv} The numerical solution $u_{\eps, n}$ given by the Strang
  splitting scheme with frequency cut-off $|k| \leq \epsilon^{-1}$
  (defined below in
  \eqref{eq:strangmollified}) 
  with time step size $\tau > 0$ provided $\tau < C \eps^{5/2 + \eta}$
    for some $\eta > 0$ on an interval of size $T = K \tau$ for some large $K$,
has a second-order error bound in
  $H^1$
\begin{equation}
  \norm{u_{\eps, n} - u(t_0 + \tau n)}_{H^1} \leq C(m_{7}, T) ( \tau^2 + \epsilon ) , 
\end{equation}
where 
\begin{equation*}
  m_N = \max_{0 \leq t \leq T} \norm{u(t)}_{H^N}. 
\end{equation*}
\end{theorem}}

\begin{remark}
  The small data local existence in $H^s$ for the cubic quasilinear nonlinear
  terms is established in Marzuola-Metcalfe-Tataru \cite{MMT3,MMT4}
  (See also the works of Poppenberg \cite{Poppenberg},
  Kenig-Ponce-Vega and Kenig-Ponce-Rolvung-Vega
  \cite{KPV2,KPV3,KPV,KPRV1,KPRV2}).  In particular, in the case of
  equation \eqref{eqn:model} in dimension $d$, there exists a local in
  time solution in $H^\sigma$ as long as $\sigma > \frac{d+5}{2}$
  provided $u_0$ is sufficiently small. Therefore, the regularity
  assumption $m_7 < \infty$ holds for sufficiently small data in $H^\sigma$ for $\sigma$ sufficiently large (and hence in $L^\infty$).  This is far from sharp however and much work must be done to explore the threshold between well-posedness and blow-up.  
\end{remark}

\begin{remark}
The interplay between $\tau$ and $\epsilon$ parameters arise only from the order $\epsilon$ bounds on the $H^1$ remainder of cutting the initial data off to frequencies below $\epsilon^{-1}$ and evolving with a mollified Schr\"odinger flow.  This will fully be addressed in Section \ref{sec:mollify}, while below we compute the dependence with respect to $\tau$.  Indeed, the authors observed numerically that instability occurs when $\tau$ is taken too large.  
\end{remark}


{We will actually prove the result for arbitrary spatial dimension $d$, since the ideas are the same for any dimension.  To start, we wish to establish the stability of the Strang splitting scheme with respect to a fixed time step.  Before we begin let us take
\begin{equation}
m_k =  {\| u(t) \|_{L^{\infty}H^k}}, \ k \leq \max \bigl(7,\frac{d+6}{2}+{\eta}\bigr),
\end{equation}
for any $\eta > 0$ such that $u$ the solution to \eqref{eq:qls} with small initial data can be defined in $H^k$ using \cite{MMT3,MMT4}.  }

We can approximate the solution through the continuous time generators of the split step equations:
\begin{align}
\label{eqn:Deltadot}
i \dot \psi & = -\Delta \psi \\
\label{eqn:Vdot}
i \dot \psi & = V[ \psi] \psi
\end{align}
where
\begin{align}\label{eqn:sftfq}
V[\psi] = |\psi|^2 + \Delta (|\psi|^2).
\end{align}
The generators of the split step method can thus be described as exponential maps of the vector fields given by 
\begin{align}
\hat{T} (\psi) & = i \Delta \psi, \label{eqn:hatT} \\
\hat{V} (\psi) & = -i V[\psi] \psi = -i [ |\psi|^2 + \Delta (|\psi|^2)]  \psi.  \label{eqn:hatV} 
\end{align}

The key estimate we will require is of the type 
\begin{equation}
\label{eqn:tri2}
\| \Delta (u v) w \|_{H^s} \leq  C  \| u \|_{H^{s+\frac{4+d}{2}+{\eta}}} \| v \|_{H^{s+\frac{4+d}{2}+{\eta}}} \| w \|_{H^s}
\end{equation}
for $s > 0$
using the $L^\infty \times L^\infty \times L^2 \to L^2$ H\"older's inequality and the Sobolev embedding for $L^\infty$.  Note, the total loss of regularity on a given component of the multilinear estimates can be reduced with other estimates such as
\begin{equation}
\label{eqn:tri1}
\| \Delta (u v) w \|_{H^s} \leq C \| u \|_{H^{s+\frac{6+d}{3}}} \| v \|_{H^{s+\frac{6+d}{3}}} \| w \|_{H^{s+ \frac{d}{3}}}
\end{equation} 
for $s \geq 0$ using $L^6 \times L^6 \times L^6 \to L^2$ H\"older's inequality and the Sobolev embedding for $L^6$ as well as moving to $L^p$ based spaces and applying the techniques of Strichartz estimates, etc.  {However, we will use  \eqref{eqn:tri2} throughout since in the mollification the loss in $\epsilon$ will be scale invariant for any choice, plus it turns out to be beneficial to have one component remain free of derivatives to take advantage of the short time gains we will observe in the Lie Theory.}

Before computing Lie derivatives, we want to understand the stability of the evolution generated by $\hat{V}$.  To do this, we study
\begin{equation}
\label{eqn:Vode}
i \dot \nu = V[ \psi ] \nu, \ \ \nu (0) = \psi.
\end{equation}
For $\psi$ sufficiently regular, it is possible to show that the evolution varies continuously with the choice of initial data in a weak topology.  In particular, given
\begin{align*}
i \dot \nu = V[ \psi ] \nu, \ \ \nu (0) = \psi, \\
i \dot \mu = V[ \phi ] \mu, \ \ \mu (0) = \phi, 
\end{align*}
by looking at the difference of these two evolutions, expanding $V[\phi] \mu - V[\psi] \nu = (V[\phi] - V[\psi]) \mu - V[\psi] (\nu - \mu)$ and applying \eqref{eqn:tri2} we have
\begin{equation}
\label{eqn:Liecont}
\| \mu (t) - \nu(t) \|_{H^s} \leq \| \psi - \phi \|_{H^s} + C_1 t \| \psi - \phi \|_{H^{s+\frac{d+4}{2}+{\eta}}} + C_2 \int_0^t   \| \mu (s) - \nu (s) \|_{H^s},
\end{equation}
for any ${\eta} >  0$ and $s > \frac{d+4}{2}$ chosen sufficiently large to control the evolution, where $C_1$, $C_2$ both depend upon $M_s = \max_{0\leq t \leq T} \{ \| \phi \|_{H^s}, \| \psi \|_{H^s} \}$.   As a result, a Gronwall type argument shows
\begin{equation}
\label{eqn:gronwall1}
\| \mu (t) - \nu(t) \|_{H^s} \leq e^{C_0 \tau} \| \psi - \phi \|_{H^{s+\frac{d+5}{2}}}
\end{equation}
where $C_0$ depends upon $m_k$.

{Unfortunately, the above estimate comes with a regularity loss
  (it requires higher Sobolev on the previous time step). As a result,
  it is not sufficient to provide an error estimate unless we assume
  that the solution is smooth.  \footnote{{The authors thank
      Ludwig Gauckler for pointing out this which leads an error in an
      earlier version of the argument.}} This can be dealt with for
  certain types of derivative nonlinearities that are not quasilinear
  however, see for instance \cite{HLR:2013}.  See also \cite{AmSi:14} for a recent analytic treatment of convergence of mollified derivative Schr\"odinger equations on the torus.}

{To resolve the issue, let us introduce a mollified equation for $u_\epsilon = G_\epsilon u_\epsilon$ given by
\begin{equation}\label{eq:mollflow}
  i u_{\eps, t} + u_{\eps, xx} = G_{\eps} \ast [\abs{u_{\eps}}^2 u_{\eps}] 
  + G_{\eps} \ast [\Delta (\abs{u_{\eps}}^2) u_{\eps}], 
\end{equation}
where $G_{\eps}$ is a smooth, compactly supported mollifier that cuts off the high frequency terms of the evolution such that
if $u \in H^s$, we have 
\begin{align}
  & \norm{G_{\eps} u - u}_{H^{s-1}} \leq C \eps \norm{u}_{H^s}, \quad \text{and} \\
  & \norm{G_{\eps} u}_{H^{s+1}} \leq C \eps^{-1} \norm{G_{\eps} u}_{H^s}.
\end{align}
Note $G_{\eps}$ is essentially a smooth cut-off in frequency space at the frequency $\Or(1/\eps)$. 
}

{As will be shown in Section~\ref{sec:mollify}, the mollified flow
  and the original flow are close for small initial data using the frequency envelope type arguments of \cite{MMT3,MMT4} that prove the high frequency terms remain small over the order $1$ lifespan. If
  $\norm{u_0}_{H^{s+3}}$ is small, then
  \begin{equation}
    \norm{u_{\eps}(t) - u(t)}_{L^{\infty}H^s} \leq C \eps.
  \end{equation}
  Hence, we will analyze the Strang splitting on the mollified flow
  \eqref{eq:mollflow}. From a point of view of the fully discretized
  flow, e.g., with pseudospectral method, the inclusion of a frequency
  cut-off is also quite natural. Actually, the numerical result in
  Section~\ref{sec:Num} can be understood as discretizations of
  \eqref{eq:mollflow} since the number of Fourier modes $N$ is fixed
  as the time step is reduced, and the numerical solution is compared
  to that with a tiny time step (but fixed spatial resolution).}

{For the mollified flow, the Strang splitting scheme converges
  with second order error.
  \begin{prop} Consider the numerical solutions $u_{\eps, n}$ given by the
    Strang splitting scheme on the mollified equation: 
    \begin{equation}\label{eq:strangmollified}
      \begin{aligned}
        & u_{\eps, n+1/2}^- = e^{\frac{i}{2} \tau \Delta} u_{\eps, n}; \\
        & u_{\eps, n+1/2}^+ = u_{\eps, n+1/2}^- \exp\bigl( - i \tau G_{\eps} (\abs{u_{\eps, n+1/2}^-}^2 + \Delta \abs{u_{\eps, n+1/2}^-}^2 )\bigr)  \\
        & u_{\eps, n+1} = e^{\frac{i}{2} \tau \Delta} ( G_\epsilon u_{\eps, n+1/2}^+). \\
      \end{aligned}
    \end{equation}
    The numerical solution converges to the solution to the mollified
    equation as $\tau \to 0$, provided $\tau < C \eps^{5/2 + \eta}$
    for some $\eta > 0$
    \begin{equation}
      \norm{u_{\eps, n} - u_{\eps}(t_0 + \tau n)}_{H^1} \leq C(m_7, T) \tau^2. 
    \end{equation}
\end{prop}}

{
\begin{remark}
In the mollified Strang splitting algorithm above, it is possible that there is small loss of $L^2$ norm conservation in the splitting scheme due to cutting off at high-frequency in the third step of the method.  However, we note that in the pseudo-spectral method, the evaluation of the product of the nonlinear phase and $ u_{\eps, n+1/2}^-$ in the middle step is done completely on the spatial side, which already includes essentially a cut-off below a given frequency scale related to the grid spacing.  Hence, in the fully discrete implementation, the $L^2$ norm is actually conserved.
\end{remark}

\begin{remark}
  Theorem~\ref{thm:conv} follows from the above and
  Proposition~\ref{prop:mollify} in Section~\ref{sec:mollify}.
\end{remark}

\begin{proof}
  The convergence proof follows a Lie theoretic idea of Lubich
  \cite{Lubich:08} for semilinear nonlinear Schr\"odinger equations.  

  Denote
  \begin{equation*}
    V_{\eps}[\psi] = G_{\eps} \ast [ \abs{\psi}^2 \psi ] + G_{\eps} \ast [ \Delta(\abs{\psi}^2) \psi]
  \end{equation*}
  and consider two flows given by 
  \begin{align*}
    & i \dot{\nu} = V_{\eps}[\psi] \nu, \ \ \nu(0) = \psi; \\
    & i \dot{\mu} = V_{\eps}[\psi] \mu, \ \ \mu(0) = \phi.
  \end{align*}
  Using essentially the same calculation leading to
  \eqref{eqn:Liecont}, we arrive at
  \begin{equation}
    \begin{aligned}
      \norm{\mu(t) - \nu(t)}_{H^s} & \leq \norm{\psi - \phi}_{H^s} +
      C_1 t \norm{G_{\eps} \ast (\psi - \phi)}_{H^{s+\frac{d+4}{2} +
          \eta}} + C_2 \int_0^t \norm{\mu(s) - \nu(s)}_{H^s} \\
      & \leq \norm{\psi - \phi}_{H^s} + C_1 t \eps^{-(d+4)/2 - \eta}
      \norm{\psi - \phi}_{H^s}  + C_2 \int_0^t \norm{\mu(s) -
      \nu(s)}_{H^s}. 
    \end{aligned}
  \end{equation}
  Note that in the last step, we have used an inverse inequality
  thanks to the frequency cut-off in $G_{\eps}$.  Therefore, as far as
  $t \leq \eps^{(d+4)/2 + \eta}$, a Gronwall type argument gives
  \begin{equation}
    \norm{\mu(t) - \nu(t)}_{H^s} \leq e^{C_0 t} \norm{\psi - \phi}_{H^s}
  \end{equation}
  where $C_0$ depends on $m_k$ for $k = s + \frac{d}2 + 2 + \eta$.

Now, to compare the full evolution to the mollified split-step method, we must compute the Lie commutators between generating vector fields.  As mollification will only reduce norms below, for simplicity, we work with continuous versions of the Strang splitting flow.  However, the reading should keep in mind using $V_\epsilon$ in place of $V$ below.  We observe
\begin{align}
[ \hat{T}, \hat{V} ] \psi & =  \Delta \left( |\psi|^2 \psi - \Delta (|\psi|^2) \psi \right) \notag \\
& \hspace{.2cm} - \left[ 2 \Delta \psi ( \bar{\psi} \psi - \psi^2 \overline{\Delta \psi} )\right]  \label{eqn:Lie} \\
& \hspace{.4cm} - \left[  \Delta ( \Delta \psi \bar{\psi}) \psi - \Delta (\psi \overline{\Delta \psi} ) \psi + \Delta (|\psi|^2 ) \Delta \psi \right].  \notag
\end{align}
Hence,
\begin{equation}
\label{eqn:commutbd1}
\| [\hat{T}, \hat{V} ] (\psi) \|_{H^1} \leq C \| \psi \|_{H^{\max(5,\frac{4+d}{2}+)}}^3.
\end{equation}
In addition, we then can easily compute
\begin{align}
\label{eqn:commutbd2}
\| [ \hat T, [ \hat{T}, \hat{V} ] ] (\psi) \|_{H^1} \leq C \| \psi \|_{H^{\max(7,\frac{6+d}{2}+)}}^3.
\end{align}
Setting the vector field $\hat H = \hat T + \hat V$, the underlying idea is that the evolution of the full quasilinear Schr\"odinger equation given by the exact evolution
\begin{equation}
\label{eqn:duhamel}
\psi( \tau) = \exp( \tau D_H) \Id (\psi_0) \,
\end{equation}
when well defined (by making the initial condition sufficiently high regularity) can be compared
through a double Duhamel expansion to the split-step generator
\begin{equation}
\label{eqn:dduhamel}
\psi_{SS} ( \tau) = \exp( \frac12 \tau D_T)  \exp( \tau D_V) \exp( \frac12 \tau D_T)  \Id (\psi_0) ,
\end{equation}
the error terms of which can be written using the Lie commutators.  Since the frequency mollifier we wish to include in the pseudo-spectral implementation commutes with the left-most $ \exp( \frac12 \tau D_T)$ iteration, we again proceed with the continuous estimates and recognize that in the end we will cut-off in frequency, which only reduces norms.

Indeed, the error estimates come from successive application of the quadrature first order error formula
\begin{equation}
\label{eqn:Peano1}
\tau f( \frac12 \tau ) - \int_0^\tau f(s) \ud s = \tau^2 \int_0^1 \kappa_1 (\theta) f' ( \theta \tau) \ud \theta
\end{equation}
and the second-order error formula
\begin{equation}
\label{eqn:Peano2}
\tau f( \frac12 \tau ) - \int_0^\tau f(s) \ud s  = \tau^3 \int_0^1 \kappa_2 (\theta) f'' ( \theta \tau) \ud \theta
\end{equation}
where $\kappa_1(\theta)$ and $\kappa_2(\theta)$ are the Peano kernels for the midpoint rule
and
\begin{equation*}  
f(s) = \exp ( ( \tau -s ) D_T) D_V \exp(s D_T) \Id (\psi_0)
\end{equation*}
and hence
\begin{align*}
f'(s) & = e^{i s \Delta} [ \hat T, \hat V] e^{i (\tau -s) \Delta} \psi_0,  \\
f''(s) & = e^{i s \Delta} [\hat T, [ \hat T, \hat V] ] e^{i (\tau -s) \Delta} \psi_0.
\end{align*}
Note, the Peano kernels are defined as the integral kernels of the linear transformation $$L: C^{k+1} [0,T] \to \RR$$ such that
\begin{equation*}
L (f) = f - \sum_{j=0}^k \frac{f^{(j)} (0) }{j!} T^j = \frac{1}{k!} \int_0^T \kappa_k (s) f^{(k+1)} (s) ds,
\end{equation*}
hence we observe that using the mid-point rule the $f'(t/2)$ term vanishes explaining why there is not a quadratic term in \eqref{eqn:Peano2}, though the expressions \eqref{eqn:Peano1} and \eqref{eqn:Peano2} can still vary due to the nature of the error term expansions in both cases.
Hence, it is essential that for the below we can prove that for our approximation we have $f(s) \in C^3$, which very much relates to the analyticity of the linear Schr\"odinger evolution kernel in the Strang splitting scheme as in particular a generic quasilinear Schr\"odinger flow cannot be shown to be more than $C^0$ by the purely dispersive techniques in \cite{MMT3,MMT4}. We will come back to this in more detail in Section~\ref{sec:Reg}.

Applying \eqref{eqn:duhamel}, \eqref{eqn:dduhamel}, \eqref{eqn:commutbd1}, \eqref{eqn:commutbd2} and \eqref{eqn:Liecont} in succession as in \cite{Lubich:08} gives
\begin{equation}
\label{eqn:quadconv}
\| u_{n,\epsilon} - u_\epsilon (t_n) \|_{H^1} \leq C (m_{K_0},T) \tau^2
\end{equation}
for $t_n = n \tau \leq T$ and $K_0 = \max (7, \frac{d + 6}{2}+\eta)$ for the mollified flow when $\tau$ is small compared to $\epsilon$.  

From the continuous point of view, in order to obtain the $\tau^2$ convergence here, it is important to compute the double commutator bound leading to \eqref{eqn:dduhamel} in order to expand out to $3$rd order in the Lie derivatives.  However, we note the same quadratic convergence would hold in $L^2$ with only $K_0 = \max (5, \frac{d + 5}{2}+\eta)$ as then the double Duhamel commutator would not be required and we would be mostly restricted by the well-posedness threshold for \eqref{eq:qls}.  
\end{proof}}

\begin{remark}
  So far we have considered the convergence of the time-splitting
    flow to the flow of the original PDE. We further discretize the
    spatial degree of freedom using a Fourier pseudo-spectral
    method. The convergence of the fully discretized scheme follows if
    we can show that the fully discretized scheme converges to the
    time-splitting flow. Though this is beyond the scope of our aim in
    this work. See for instance \cite{Gauckler:11, ShenWang:13,
      Thalhammer1,Thalhammer2} for analysis of fully discretized
    scheme for semilinear Schr\"odinger equations.  { The importance of analysis of the fully discretized scheme moving forward is quite clear from the necessity of mollifying the Strang splitting algorithm here.  The authors hope to consider this more carefully in future work.}
\end{remark}

{
\begin{remark}
Since the mollified equation becomes effectively semi-linear, one could pose the question as to whether or not much of the quasilinear analysis presented here is necessary for proof of convergence or if the semilinear tools from \cite{Lubich:08} for instance could be applied.  Actually, one can apply the semilinear techniques to the mollified flow, however the existence time of the model or the initial data would become exponentially small depending upon the $\epsilon$ threshold in the frequency cut-off.  Hence , using the quasilinear flow estimates is quite important in order to get uniform bounds.  
\end{remark}}

\section{Stability and instability of the numerical scheme}\label{sec:Stab}

As discussed in Section~\ref{sec:Num}, for large data, we observe blow-up behavior in the numerical study. In this section, we will investigate the numerical instability of the scheme, which will shed some light on the blow-up behavior. 

\subsection{Linear stability analysis for wave train} 

For the uniform wave trains solution \eqref{eq:planewave}, we study
the stability for perturbations around the solution. Consider a
perturbed solution of the form
\begin{equation}
  u(x, t) = u_0(x, t) ( 1 + \veps(x, t)), 
\end{equation}
where  $u_0$ is the plane-wave solution $a e^{i(kx-\omega t)}$ and $\abs{\veps}^2 \ll 1$. To the leading order, we get 
\begin{equation}
  i \veps_t + 2 i k \veps_x + \veps_{xx} = \abs{a}^2 ( \veps + \wb{\veps}) 
  + \abs{a}^2 (\veps + \wb{\veps})_{xx}. 
\end{equation}
Let us expand $\veps$ in Fourier series (with $\xi_n = n$ for $\veps$
periodic on $[0, 2\pi]$):
\begin{equation}
  \veps(x, t) = \sum_{n=-\infty}^{\infty} \wh{\veps}_n(t) \exp(i \xi_n x). 
\end{equation}
The equation of $\veps$ can then be written as a system of ODEs,
\begin{equation}
  \frac{\rd}{\rd t} 
  \begin{pmatrix}
    \wh{\veps}_n \\
    \wh{\wb{\veps}}_{-n}
  \end{pmatrix}
  = G_n   \begin{pmatrix}
    \wh{\veps}_n \\
    \wh{\wb{\veps}}_{-n}
  \end{pmatrix}, 
\end{equation}
where 
\begin{equation}
  G_n = i 
  \begin{pmatrix}
    - 2 k \xi_n - \abs{\xi_n}^2 - \abs{a}^2 + \abs{a}^2 \abs{\xi_n}^2
    & - \abs{a}^2 + \abs{a}^2 \abs{\xi_n}^2 \\
    \abs{a}^2 - \abs{a}^2 \abs{\xi_n}^2 & - 2 k \xi_n + \abs{\xi_n}^2
    + \abs{a}^2 - \abs{a}^2 \abs{\xi_n}^2
  \end{pmatrix}
\end{equation}
The eigenvalues of $G_n$, $\lambda_n$ is given by 
\begin{equation}
  \lambda_n = - 2 ik \xi_n \pm \abs{\xi_n} \sqrt{-\abs{\xi_n}^2 - 2 \abs{a}^2 
    + 2 \abs{a}^2 \abs{\xi_n}^2}
\end{equation}
The solution becomes unstable if one of the eigenvalues has a positive
real part or, equivalently,
\begin{equation}
  2\abs{a}^2 \abs{\xi_n}^2 - 2 \abs{a}^2 - \abs{\xi_n}^2 > 0 
\end{equation}
A sufficient condition for stability is 
\begin{equation}
  \abs{a} \leq \sqrt{2}/2. 
\end{equation}
Note that the stability threshold $\sqrt{2}/2$ agrees with the
  numerical observations in Section~\ref{sec:Num}.

\subsection{Linear Stability Analysis for the Strang splitting algorithm}

To study the stability of general initial data, where an explicit
  solution is not available, we linearize around a solution to
  \eqref{eq:qls}.  Locally, the linear instability is
essentially equivalent to the plane-wave instability observed in
Section \ref{s:pw}, and analyzed in the previous subsection.  The key
 observation is that the instability occurs for all $k$.  This
has been done in \cite[Equation $(8)$]{LPT}, where one observes that
perturbation around a solution $u = w + z$ leads to
\begin{equation}
\label{linearized}
z_t = a \left[ M_w \Delta z + G_w \nabla z + H_w z \right]  + f(t), \quad z(0) = g,
\end{equation}
where, for $w = w_1 + i w_2$, we observe
\begin{equation*}
  M_w = \left[ \begin{array}{cc}
2 w_1 w_2 & 2 w_2^2 - 1 \\
1- 2 w_1^2 & -2 w_1 w_2
\end{array} \right],
\end{equation*}
which has determinant $1- 2 |w|^2$.  The matrix functions $G_w$ and $H_w$ come from the linearization and will be expressed in full in \eqref{eq:linqls} below. Using this linearization and a Fr\'echet based iteration argument in the space $H^\infty$, the authors then show local well-posedness for small data solutions to equations of the form \eqref{eq:qls}.

To understand the instability of the numerical scheme, we linearize
the discretized Strang splitting algorithm and show that the linear
instability threshold in the continuous problem exists in the
discretized version as well.  Letting $u = w + z$ for some solution
$w$ of \eqref{eqn:thinfilm}, and generally writing $h(x,t) = h_1
(x,t) + i h_2 (x,t)$ for any complex function $h$, we have that
the linearized continuous PDE \eqref{linearized} takes the form
(see also \cite{LPT})
\begin{align}
\label{eq:linqls}
&\left[ \begin{array}{c}
 z_1 \\ z_2 \end{array} \right]_t = \left[ \begin{array}{cc} 
 2 w_1 w_2 \Delta & (2w_2^2 -1)\Delta \\
(1 - 2 w_1^2)  \Delta & -2 w_1 w_2 \Delta  \end{array} \right] \left[ \begin{array}{c}
z_1 \\ z_2 \end{array} \right] + \left[ \begin{array}{cc} 
2 w_2^2 + 2 w_1 w_2 & w_1^2 + w_2^2 \\
-[3 w_1^2 + w_2^2] & -2 w_1 w_2  \end{array} \right] \left[ \begin{array}{c}
z_1 \\ z_2 \end{array} \right]   \\
& \hspace{1.5cm} + \left[ \begin{array}{cc} 
 4 w_2 \nabla w_1 \cdot \nabla  & 4 w_2 \nabla w_2 \cdot \nabla\\
-\left[  4 w_1 \nabla w_1 \cdot \nabla \right] & -\left[4 w_1 \nabla w_2 \cdot \nabla \right]  \end{array} \right] \left[ \begin{array}{c}
z_1 \\ z_2 \end{array} \right]  \notag \\
& \hspace{1.5cm} +  \left[ \begin{array}{cc} 
2 w_2 \Delta w_1 &  \sum_{j=1}^2 2 \nabla \cdot (w_j \nabla w_j)  ) + 2w_2 \Delta w_2 \\
 -\left[  \sum_{j=1}^2 2 \nabla \cdot (w_j \nabla w_j)  + 2w_1 \Delta w_1  \right]  &- 2 w_1 \Delta w_2   \end{array} \right] \left[ \begin{array}{c}
z_1 \\ z_2 \end{array} \right].  \notag
\end{align}
We then have
\begin{equation*}
M_w = \left[ \begin{array}{cc}
2 w_1 w_2 & 2 w_2^2 - 1 \\
1- 2 w_1^2 & -2 w_1 w_2
\end{array} \right], \quad G_w =  \left[ \begin{array}{cc}
4 w_2 \nabla w_1^T  & 4 w_2 \nabla w_2^T \\
-4 w_1 \nabla w_1^T&  -4 w_1 \nabla w_2^T
\end{array} \right], 
\end{equation*}
\begin{align*} H_w & = \left[ \begin{array}{cc}
2 w_2^2 + 2 w_1 w_2  + 2 w_2 \Delta w_1 & H_{12}  \\
-H_{21} & -[ 2 w_1 w_2 + 2 w_1 \Delta w_2]
\end{array} \right] ,
\end{align*}
where
\begin{align*}
H_{12} & = w_1^2 + w_2^2 + 2 w_1 \Delta w_1 + 2 \nabla w_1 \cdot \nabla w_1 + 4 w_2 \Delta w_2 + 2 \nabla w_2 \cdot \nabla w_2, \\
H_{21} & = 3 w_1^2 + w_2^2+ 4 w_1 \Delta w_1 + 2 \nabla w_1 \cdot \nabla w_1 + 2 w_2 \Delta w_2 + 2 \nabla w_2 \cdot \nabla w_2 + 4 w_1 \nabla w_1.
\end{align*}

We wish to compare the linearization of the full PDE to the discretized linearization of the form
\begin{align}
  & z_{n+1/2}^- = e^{\frac{i}{2} \tau \Delta} z_n; \notag \\
  & z_{n+1/2}^+ = \exp\left(-i \tau
    \bigl(\abs{w_{n+1/2}^-}^2 + 
    \Delta (\abs{w_{n+1/2}^-}^2) \bigr)\right) \left( Id -i \tau \tilde{H}_{w_{n+1/2}^-}   \right) z_{n+1/2}^-  ;  \label{eqn:linsplitting} \\
  & z_{n+1} = e^{\frac{i}{2} \tau \Delta} z_{n+1/2}^+, \notag
\end{align}
where, taking $w_{n+1/2}^- = w_1 + i w_2$, we have
\begin{align*} \tilde{H}_{w_{n+1/2}^-} & =  \left[\begin{array}{cc} 
 2 w_1 w_2 \Delta & 2w_2^2 \Delta \\
- 2 w_1^2  \Delta & -2 w_1 w_2 \Delta \end{array} \right]  + \left[ \begin{array}{cc} 
2 w_2^2 + 2 w_1 w_2 & w_1^2 + w_2^2 \\
-[3 w_1^2 + w_2^2] & -2 w_1 w_2  \end{array} \right]    \\
& \hspace{.5cm} + \left[ \begin{array}{cc} 
 4 w_2 \nabla w_1 \cdot \nabla & 4 w_2 \nabla_n w_2 \cdot \nabla \\
-\left[  4 w_1 \nabla w_1 \cdot \nabla \right] & -\left[4 w_1 \nabla w_2 \cdot \nabla \right]  \end{array} \right]   \notag \\
& \hspace{.5cm} +  \left[ \begin{array}{cc} 
2 w_2 \Delta w_1 &   \sum_{j=1}^2 2 \nabla \cdot (w_j \nabla w_j)  + 2w_2 \Delta w_2 \\
-\left[ \sum_{j=1}^2 2 \nabla \cdot (w_j \nabla w_j)  + 2w_1 \Delta w_1  \right] & -\left[ 2 w_1 \Delta w_2  \right]  \end{array} \right] . \notag
\end{align*}

To address the linear stability of the Strang splitting scheme, consider a linearly unstable mode corresponding to \eqref{eq:linqls} such that $z$ is regular.  The linearized splitting scheme to the leading order works as 
\begin{align}
  & z_{n+1/2}^- = \left(Id + \frac{1}{2} \begin{pmatrix} & - \tau \Delta \\ \tau \Delta & 0 \end{pmatrix}\right)  z_n; \notag \\
  & z_{n+1/2}^+ = \left(Id +  \tau \begin{pmatrix}
 2 w_1 w_2 \Delta_n & 2w_2^2 \Delta_n \\
- 2 w_1^2  \Delta_n & -2 w_1 w_2 \Delta_n  \end{pmatrix}\right) z_{n+1/2}^-  ; \label{eq:linss}  \\
  & z_{n+1} = \left(Id + \frac{1}{2} \begin{pmatrix} & - \tau \Delta \\ \tau \Delta & 0 \end{pmatrix}\right) z_{n+1/2}^+, \notag
\end{align}
and hence 
\begin{equation}
\label{eq:ssheatflow}
   z_{n+1} = z_n  + \tau \begin{pmatrix} 
 2 w_1 w_2 \Delta & (2w_2^2 -1)\Delta \\
(1 - 2 w_1^2)  \Delta & -2 w_1 w_2 \Delta  \end{pmatrix} z_n + \mathcal{O} (\tau^2).
\end{equation}
As a result, if $w_1$ and $w_2$ are constant (as we can assume only locally with any accuracy) and $2 |w|^2 -1 > 0$, we observe that each Fourier mode $z_{n+1,k}$ can be approximated by the linearized dynamical system
\begin{equation}
   z_{n+1,k} = z_{n,k} + \tau \begin{pmatrix} 
 -2 w_1 w_2 k^2 & -(2w_2^2 -1) k^2 \\
(2 w_1^2-1) k^2 & 2 w_1 w_2 k^2  \end{pmatrix} z_{n,k} ,
\end{equation}
which has eigenvalues $1 \pm \tau k^2 \sqrt{ 2 |w|^2 -1}$ and hence would clearly grow exponentially for $|w| > \sqrt{2}/2$.  For sufficiently large $k$, we observe that all the Fourier modes of $w_1$ and $w_2$ are small perturbations and hence exponential growth will occur just as in the backwards heat flow generated from the continuous approximation.  Indeed, for $\tau$ sufficiently small, we have that the linearized Strang splitting flow well approximates the unstable backwards heat flow and hence displays linear instability.  However, of course, the nonlinear effects are ignored in this computation.  

To make this more precise, we observe from \cite{MMT4} for the full PDE model \eqref{eqn:model}, that we can construct initial data for \eqref{eqn:thinfilm} having $L^\infty$ norm larger than $\sqrt{2}/2$ but sufficiently localized in frequency such that the solution exists locally in time.  A simple example of a solution with a global existence time and initial $L^\infty$ norm larger than $\sqrt{2}/2$ is an exact plane-wave solution for the periodic problem.  However, perturbations of such exact solutions are still linearly unstable as calculated in Section $4.1$.  The backward heat equation that represents the linearization of the continuous model exists locally in time when considering frequency localized data; however, this time scale will depend on the frequency cut-off and potentially be quite short given the nonlinear interactions.  Using Theorem \ref{thm:conv}, we observe that, in the semi-discretized equation, choosing $\tau \ll 1$ sufficiently small compared to the scale of local existence for \eqref{eqn:thinfilm}, the numerical solution computed using the Strang splitting scheme is $\mathcal{O} (\tau^2)$ provided the solution is sufficiently regular.  Of course, from \eqref{eq:ssheatflow}, on a single time-step, the Strang splitting solution has a polynomial instability.  However, the linearized equations in \eqref{eq:linqls} give exponential growth dynamics for the full solution on the time scale of local existence.  Since, on the scale of existence, the numerical solution remains $\mathcal{O} (\tau^2)$, the linear instability is inherited by the numerical solution over repeated iterations of the time step.

\section{Regularity of the Time Evolution}
\label{sec:Reg}

In the analysis of the convergence of the Strang splitting scheme, we have used the analyticity of the linear Schr\"odinger evolution. However, this is not the case in general for the quasilinear Schr\"odinger evolution. The Strang splitting scheme actually regularizes the time flow of the original PDE. In this section, we give some further discussion for the regularity of the time evolution.

{  Let us show that the solution map of a quasilinear Schr\"odinger
evolution is continuous in time for \eqref{eq:qls}.  The continuity partially hinges upon
the proof of uniqueness for the evolution.  In
particular, take two solutions to a more general quasilinear model of the form \eqref{eqn:quasiquad}, say $u_1$ and $u_2$.}
Setting $v = u_1 - u_2$ and linearizing, we have an
equation of the form
\begin{equation*}
\left\{ \begin{array}{l}
i v_t + a^{jk} (u ) \p_j \p_k v + V \nabla v + W v  = 
0 ,  \\
v(0,x) = u_1 (0) - u_2 (0) 
\end{array} \right. 
\end{equation*}
with 
$$ V = V(u_1, \nabla u_1, u_2, \nabla u_2), \quad W = h (u_1, \nabla u_1, u_2, \nabla u_2) + g(u_1, u_2 )\nabla^2 u_1$$
for functions $V$, $h$ and $g$ related to the Taylor expansion of the metric and the nonlinearity. 
Then, to solve this linear equation, we use \cite[Proposition $5.1$]{MMT4} (see also \cite[Proposition $5.2$]{MMT3}) to show that the weak Lipschitz bound 
\begin{equation}\label{weak-lip}
\|  v \|_{L^\infty H^{\sigma}} \lesssim \|  v(0) \|_{H^{\sigma}}
\end{equation}
holds for any $0 \leq \sigma \leq s-1$ via energy estimates on the linearized equation, where we recall that the initial condition lies in $H^s$ for $s > \tfrac{d + 5}{2}$.  
However, in the well-posedness result for the linearized version of \eqref{eq:qls} (see Proposition $5.1$ of \cite{MMT4}, for instance) we have at most continuity of the solution map with respect to time in the $H^s$ norm.   

The key ideas of the proof follow from the theory of frequency envelopes as discussed in both \cite[Sections $2$ and $5$]{MMT3,MMT4}, where it is proven that the size of a dyadic frequency component of the solution to \eqref{eq:qls} in a natural energy space is bounded by a uniform constant times the corresponding dyadic frequency component of the initial data in $H^s$.  To be more precise, we shall use a
Littlewood-Paley decomposition of the spatial frequencies,
\[\sum_{i=0}^\infty S_i(D)=1,\]
where $S_i$ localizes to frequency $\abs{\xi} \in [2^{i-1}, 2^{i+1}]$ for $i>0$ and to frequencies $|\xi|\le 2$ for $i=0$. By a frequency envelope, we recall from \cite[Section $2.4$]{MMT3} that we mean that given a translation invariant space $U$ such that
\[
\|u\|_{U}^2 \sim \sum_{k=0}^\infty \|S_k u\|_{U}^2, 
\]
a frequency envelope for $u$ in $U$ is a positive sequence $a_j$ so
that
\begin{equation} 
\| S_j u\|_{U} \leq a_j \|u\|_{U}, \qquad 
 \sum a_j^2 \approx 1.
\end{equation}
We say that a frequency envelope is admissible if $a_0\approx 1$ and 
it is slowly varying,
\[
a_j \leq 2^{\delta |j-k|} a_k, \qquad j,k \geq 0 , \qquad 0 < \delta
\ll 1.
\]
An admissible frequency envelope 
always exists, say by 
\begin{equation}\label{freqEnv}
a_j = 2^{-\delta j} + \|u\|^{-1}_{U} \max_k 2^{-\delta |j-k|} \|S_k u\|_{U}.
\end{equation}
Abusing notation and avoiding for simplicity the atomic space formulations in \cite{MMT3,MMT4}, we rely upon a uniform bound over the evolution such that effectively
\begin{equation}\label{mainb}
\| u\|_{L^\infty H^s} \lesssim \|u_0\|_{H^s}.
\end{equation}
We note that the $L^\infty H^s$ norm appearing in the estimate here is due to the cubic interactions in the nonlinearity and the compactness of our domain, otherwise one must enforce further summability as in \cite{MMT3}.  A key estimate is the following proposition.

\begin{prop}[Proposition $5.3$, \cite{MMT3}; Proposition $5.4$, \cite{MMT4}]
\label{envelopes}
Let $u$ be a small data solution to \eqref{eq:qls}, which satisfies
\eqref{mainb}. Let $\{a_j\}$ be an admissible frequency envelope for the initial
data $u_0$ in $H^s$. Then $\{a_j\}$ is also a frequency envelope for $u$ 
in $L^\infty H^s$.
\end{prop}

Once we have Proposition \ref{envelopes}, the continuity of the solution map can be established as is Section $5.7$ of \cite{MMT3}.  Namely, we consider a sequence of initial data $\{ u_0^{n} \} \to u_0$ in $H^s$.  Frequency envelope bounds can then be chosen such that there exists a uniform $N_\epsilon$ for which
\begin{equation*}
\| a_{N_\epsilon}^{(n)} \| \leq \epsilon
\end{equation*}
for all $n$, which gives a uniform upper bound by Proposition \ref{envelopes} on the high frequencies of each corresponding solution $u^{(n)}$ to \eqref{eq:qls} with initial data $u_0^{(n)}$ in the $L^\infty H^s$ norm.  Separating into low and high frequencies and using the smallness of the high frequencies and the uniform convergence in weaker Sobolev norms provided by \eqref{weak-lip}, the result follows.  

Let us emphasize however that we generally gain no more than continuity of the solution map from such arguments.  Hence, in order to accurately compare the flow of the full solution map defined by \eqref{eq:qls} and that of the Strang splitting method, we rely on differentiating the equation and the balancing of spatial and time regularity, as in Section \ref{sec:Lie}. 

{

\section{Convergence of the Pseudospectral Flow}
\label{sec:mollify}

In this section, we address the closeness of the flow from the continuous equation to a mollified equation representing the effects of a full pseudo-spectral discretization scheme.  Namely, we take 
\begin{equation}
\label{qls:moll}
i u_{\epsilon,t} + u_{\epsilon,xx} =  G_\epsilon *[  | u_\epsilon |^2 u_\epsilon ] + G_\epsilon *[  \Delta ( |u_\epsilon |^2) u_\epsilon],
\end{equation}
for $G_\epsilon \in C^\infty_c$ a smooth, compactly supported mollifier such that if $u \in H^s$, we have
\begin{equation*}
  \| G_\epsilon u - u \|_{H^{s-1}} \leq C \eps \norm{u}_{H^{s}}
\end{equation*}
as $\veps \to 0$.  Note, this is essentially an exponentially decaying
cut-off in frequency space $S_{<N_{\epsilon}}$ for $N_\epsilon = \Or(1
/ \epsilon)$. We wish to compare this to the evolution of
\eqref{eq:superfluid}.

\begin{prop}\label{prop:mollify}
For $ \|u_0 \|_{H^s} \ll 1$ sufficiently small with $s > 3$, \eqref{qls:moll} has a solution that exists for time $1$ and remains sufficiently small.  In addition, if $u$ solves \eqref{eq:superfluid} we have $\| u - u_\epsilon \|_{L^\infty ([0,1] \times H^\sigma)} = \Or(\eps)$ as $\epsilon \to 0$ for all $\sigma < s-3$. 
\end{prop}

\begin{proof}
Intuitively, we rely on the paradifferential scheme providing the frequency envelope bounds of \cite{MMT4}, which states that at least for small enough data with enough regularity, on time $1$ intervals the high frequencies do not change the flow very much.  In particular, we use the fact that the flow of both equations is well-defined in $H^s$ for $s > \frac{d+5}{2}$.  

Let us for the sake of completeness briefly review of Paradifferential Estimates from \cite{MMT4}.  We are interested for our particular numerical purposes in \eqref{eq:qls}, but as the results are also true in higher dimensions, let us work  with a more general quasilinear equation of the form \eqref{eqn:quasiquad}.  We use a Picard iteration scheme to boil down finding a solution to solving the linear problem
\begin{eqnarray}
\label{lin1}
\begin{cases}
  (i\partial_t + \partial_k a^{kl} (w) \partial_l)u + V\nabla u + Wu=H,\\
u(0)=u_0
\end{cases}
\end{eqnarray}
and 
\begin{eqnarray}
\label{lin2}
\begin{cases}
  (i\partial_t + \partial_k a^{kl} (w) \partial_l)u + V\nabla u=H,\\
u(0)=u_0
\end{cases}
\end{eqnarray}
under the assumption that
\[g^{kl}-\delta^{kl} = h^{kl}(w(t,x))  \]
where $h(z)=O(|z|^2)$ near $|z|=0$ with $w$ a small function in an energy space and $H$ a generic forcing term that is small in the dual to that energy space at the moment.  We include $H$ such that the error term from frequency cut-offs can be included below.  

We want to use a paradifferential scheme such that at $u_j$ at frequency $j$ solves
\[
\begin{cases}
 (i\partial_t + \partial_k a^{kl}_{<j-4}\partial_l)u_j = G_j + H_j,\\
u_j(0)=u_{0j},
\end{cases}
\]
where $a^{kl}_{<j-4} = S_{< j-4} a^{kl}$ is cut-off to slightly lower frequencies and hence
\[G_j = -S_j\partial_k g^{kl}_{>j-4}\partial_l u - [S_j, \partial_k
g^{kl}_{<j-4}\partial_l]u - S_j V \nabla u - S_j Wu.\]
Then, we construct a full solution by summing up in frequency.  

Applying the general energy estimates from Proposition $4.1$ to each of these
equations, we see that
\[\|u\|^2_{l^2 X^\sigma} \lesssim \|u_0\|^2_{H^\sigma} +
\|H\|_{l^2Y^\sigma}^2 + \sum_j \|G_j\|^2_{l^2Y^\sigma}.\]
If $W=0$, we can take $\sigma = s$, otherwise, we work with $\sigma = s-1$.  However, these estimates are strong enough to give a bootstrapping argument.  The spaces $l^2 X^\sigma$, $l^2 Y^\sigma$ here require using smoothing properties of the linear Schr\"odinger equation and will not be discussed in detail here.  See \cite{MMT4} for more details on their construction.   

Convergence estimates in Sobolev spaces follows directly from energy estimates for the truncated equation and the frequency envelope analysis in Section \ref{sec:Reg} on solutions to \eqref{eq:superfluid}.   Note that we are make no claims that our convergence estimates for the pseudospectral scheme are sharp, and in fact being more careful with convergence estimates above might improve future results.  Since we are largely worried about the $L^2$ and $H^1$ convergence, there is a relatively simple approach inspired by Propositions $5.1$ and $5.2$ from \cite{MMT4} that give a frequency envelope decomposition for the solution $u$ of \eqref{eq:superfluid}.  We observe 
\begin{align*}
& \| u - u_\epsilon \|_{L^\infty ([0,1] \times H^{\sigma})} \leq \\
& \hspace{1cm} C (\| (1- G_\epsilon) u_0 \|_{H^\sigma} + \| | u|^2 u - G_\epsilon ( | G_\epsilon u |^2 G_\epsilon u ) \|_{H^\sigma} + \| (| u|^2)_{xx} u - G_\epsilon ( (| G_\epsilon u |^2 )_{xx} G_\epsilon u ) \|_{H^\sigma} \\
& \hspace{2cm} \leq C(\norm{u}_{L^{\infty}H^{\sigma + 3}})  \epsilon, 
\end{align*}
where we have emphasized the dependence on the constant in the final
inequality on $\norm{u}_{L^{\infty}H^{\sigma + 3}}$.  The convergence
is easily controlled using to the frequency envelopes of $u$ and hence
the nonlinear expressions of $u$ using that $s > d/2$ and the
smallness of $u$.  As a result, if the initial datum has a small $\norm{\cdot}_{H^{\sigma+3}}$ norm, $\norm{u}_{L^{\infty}H^{\sigma + 3}}$ is controlled, and hence the difference between $u$ and $u_{\epsilon}$. \end{proof}


}

\section{Conclusion}

In this work, we have studied the Strang splitting scheme for
quasilinear nonlinear Schr\"odinger equations. The splitting scheme is
proved to have second order convergence for small data. We further
investigate the regularity of the time flow and the instability of the
numerical scheme which leads to blow-ups observed numerically.

Our work is motivated by numerical approaches towards time-dependent
density functional theory computations as discussed in the
documentation of the software package \textsf{Octopus}
\footnote{\url{http://www.tddft.org/programs/octopus/wiki/index.php/Main_Page}}
and also \cite{CMR:2004} and references therein. The mathematical
analysis and numerical schemes for time-dependent density functional
theory will be an exciting direction to explore in the future.

\bibliographystyle{plain}

\begin{thebibliography}{XX}

\bibitem{ADKM} G.D. Akrivis, V.A. Dougalis, O.A. Karakashian, and W.R. McKinney, {\em Numerical approximation of blow-up of radially symmetric solutions of the nonlinear Schr\"odinger equation}, SIAM J. Sci. Comput., 25(1), 186--212, 2003.

\bibitem{AmSi:14} {D. Ambrose and G. Simpson: {\em Local existence theory for derivative nonlinear Schr\"odinger equations with non-integer power nonlinearities}, preprint available at arXiv:1401.7060, 2014.}


\bibitem{ABB:13} X. Antoine, W. Bao, and C. Besse, {\em Computational methods for the dynamics of the nonlinear Schr\"odinger/Gross-Pitaevskii equations}, Comput. Phys. Commun., 184(12), 2621--2633, 2013.

\bibitem{AppelGross:02} H. Appel and E.K.U. Gross, {\em Static and time-dependent many-body effects via  density-functional theory},
  {Quantum Simulations of Complex Many-Body Systems: From Theory to Algorithms}, John von Neumann Institute for Computing Press NIC Series, 10, 255--268, 2002.

\bibitem{AscherRuuthWetton:95} U.M. Ascher, S.J. Ruuth, and B.T.R. Wetton, {\em Implicit-explicit methods for time-dependent partial differential equations}, SIAM. J. Numer. Anal., 32, 797--823, 1995.

\bibitem{BaoCai:13} W. Bao and Y. Cai, {\em Mathematical theory and numerical methods for Bose-Einstein condensation}, Kinet. Relat. Models, 6(1), 1--135, 2013.

\bibitem{BaoJakschMarkowich:03} W. Bao, D. Jaksch, and P.A. Markowich, {\em Numerical solution of the Gross-Pitaevskii equation for Bose-Einstein condensation}, J. Comput. Phys., 187, 318--342, 2003.

\bibitem{BaoJinMarkowich:03} W. Bao, S. Jin, and P.A. Markowich, {\em  Numerical study of time-splitting spectral discretizations of nonlinear {S}chr\"odinger equations in the semiclassical regimes},
SIAM J. Sci. Comput., 25, 27--64, 2003.

\bibitem{BaoMauserStimming:03} W. Bao, N. Mauser, and H.P. Stimming, {\em Effective one particle quantum dynamics of electrons: A numerical study of the Schr\"odinger-Poisson-X$\alpha$ model}, Commun. Math. Sci., 1, 809--828, 2003.

\bibitem{BaoShen:05} W. Bao and J. Shen, {\em A fourth-order time-splitting Laguerre-Hermite pseudospectral method for Bose-Einstein condensates}, SIAM J. Sci. Comput., 26, 2010--2028, 2005.


\bibitem{BCOR:2009} S. Blanes, F. Casas, J.A. Oteo, and J. Ros, {\em The Magnus expansion and some of its applications}, Phys. Rep., 470, 151--238, 2009.

\bibitem{CMR:2004} A. Castro, M.A.L. Marques, and A. Rubio, {\em Propagators for the time-dependent Kohn-Sham equations}, J. Chem. Phys., 121(8), 3425--3433, 2004.

\bibitem{Chin:07} S.A. Chin, {\em Higher-order splitting algorithms for solving the nonlinear Schr\"odinger equation and their instabilities}, Phys. Rev. E, 76, 056708, 2007.

\bibitem{CoxMatthews:02} S.M. Cox and P.C. Matthews, {\em Exponential time differencing for stiff systems}, J. Comput. Phys., 176, 430--455, 2002.

\bibitem{dBHS} A. De Bouard, N. Hayashi, and J.C. Saut, {\em Global existence of small solutions to a relativistic nonlinear Schr\"odinger equation}, Comm. Math. Phys., 189,
  73--105, 1997.

\bibitem{dBHNS} A. De Bouard, N. Hayashi, P.I. Naumkin, and J.C. Saut,
  {\em Scattering problem and asymptotics for a relativistic nonlinear
    Schr\"odinger equation}, Nonlinearity, 12, 1415--1425, 1999.

\bibitem{DescombesThalhammer:13} S. Descombes and M. Thalhammer, {\em The Lie-Trotter splitting for nonlinear evolutionary problems with critical parameters: a compact local error representation and application to nonlinear Schr\"odinger equations in the semiclassical regime}, IMA J. Numer. Anal., 33, 722--745, 2013.

\bibitem{Faou:09} E. Faou, V. Gradinaru, and Ch. Lubich, {\em Computing semiclassical quantum dynamics with Hagedorn wavepackets}, SIAM J. Sci. Comput., 31, 3027--3041, 2009.

\bibitem{Gauckler:11} L. Gauckler, {\em Convergence of a split-step Hermite method for the Gross-Pitaevskii equation}, IMA J. Numer. Anal., 31, 396--415, 2011.

\bibitem{HardinTappert:73} R.H. Hardin and F.D. Tappert, {\em Applications of the split-step Fourier method to the numerical solution of nonlinear and variable coefficient wave equations}, SIAM Rev. 15, 423, 1973.

\bibitem{Thalhammer2} H. Hofst\"atter, O. Koch, and M. Thalhammer, {\em Convergence analysis of high-order time-splitting pseudo-spectral methods for rotational Gross–Pitaevskii equations}, Numer. Math., 127(2), 315--364, 2014.


\bibitem{HLR:2013} {H. Holden, C. Lubich and N.H. Risebro, {\em Operator splitting for partial differential equations with Burgers nonlinearity}, Math. Comp., 82(281), 173--185, 2013.}

\bibitem{HMZ1} J. Holmer, J. Marzuola, and M. Zworski, {\em Soliton splitting by external delta potentials},  J. Nonlinear Sci., 17(4), 349--367, 2007.

\bibitem{JahnkeLubich:00} T. Jahnke and Ch. Lubich, {\em Error bounds for exponential operator splittings}, BIT, 40, 735--744, 2000.

\bibitem{Jin:12} S. Jin, {\em Schr\"odinger equation: Computation}, Encyclopedia of Applied and Computational Mathematics, B. Engquist ( ed.), to appear.

\bibitem{JMS:11} S. Jin, P.A. Markowich, and C. Sparber, {\em Mathematical and computational methods for semiclassical Schr\"odinger equations}, Acta Numer., 20, 121--209, 2011.

\bibitem{kt:2005} A.K.~Kassam and L.N. Trefethen, {\em Fourth-order time-stepping for stiff PDEs},  SIAM J. Sci. Comput., 26, 1214--1233, 2005.

\bibitem{KPV2} C.E. Kenig, G. Ponce, and L. Vega, {\em Small
    solutions to nonlinear Schr\"odinger equations},  Ann. Inst. H. Poincar\'e Anal. Non Lin\'eaire, 10, 255--288, 1993.

\bibitem{KPV3} C.E. Kenig, G. Ponce, and L. Vega, {\em Smoothing
    effects and local existence theory for the generalized nonlinear
    Schr\"odinger equations}, Invent. Math., 134, 489--545, 1998.

\bibitem{KPV} C.E. Kenig, G. Ponce, and L. Vega, {\em The Cauchy problem for
quasi-linear Schr\"odinger equations}, Invent. Math., 158, 343--388, 2004.

\bibitem{KPRV1} C.E. Kenig, G. Ponce, C. Rolvung, and L. Vega, {\em The general quasilinear ultrahyperbolic Schringer equation},
  Adv. Math., 206(2), 402--433, 2006.

\bibitem{KPRV2} C.E. Kenig, G. Ponce, C. Rolvung, and L. Vega, {\em  Variable coefficient Schr\"odinger flows for ultrahyperbolic
    operators}, Adv. Math., 196(2), 373--486, 2005.

  \bibitem{Krasny:86} R. Krasny, {\em A study of singularity formation in a vortex sheet by the point-vortex approximation}, J. Fluid Mech., 167, 65--93, 1986.

\bibitem{Kurihara} S. Kurihara, {\em Large-amplitude quasi-solitons in superfluid films}, J. Phys. Soc. Japan, 50, 3262--3267, 1981.

\bibitem{LPT} H. Lange, M. Poppenberg, and H. Teismann, {\em Nash Moser methods for the solution of quasilinear Schr\"odinger equations}, Comm. Part. Diff. Eqs., 24(7-8), 1399--1418, 1999.

\bibitem{Lubich:08} Ch. Lubich, {\em On splitting methods for Schr\"odinger-Poisson and cubic nonlinear Schr\"odinger equations},  Math. Comp., 77(264), 2141--2153, 2008.

\bibitem{Magnus:54} W. Magnus, {\em On the exponential solution of differential equations for a linear operator}, Comm. Pure Appl. Math., 7, 649--673, 1954.

\bibitem{MMT3} J. Marzuola, J. Metcalfe, and D. Tataru, {\em Quasilinear Schr\"odinger equations I: Small data and quadratic interactions},
Adv. Math., 231(2), 1151--1172, 2012.

\bibitem{MMT4} J. Marzuola, J. Metcalfe, and D. Tataru, {\em Quasilinear Schr\"odinger equations II: Small data and cubic nonlinearities},
 Kyoto J. Math., 54(3), 529--546, 2014.

\bibitem{McLachlanQuispel:02} R.I. McLachlan and G.R.W. Quispel, {\em Splitting methods}, Acta Numer., 11, 341--434, 2002.

\bibitem{PathriaMorris:87} D. Pathria and J.L. Morris, {\em Pseudo-spectral solution of nonlinear Schr\"odinger equations}, J. Comput. Phys., 87, 108--125, 1990.

\bibitem{PerezGarciaLiu:03} V.M. P{\'e}rez-Garc{\'\i}a and X. Liu, {\em Numerical methods for the simulation of trapped nonlinear Schr\"odinger systems}, Appl. Math. Comput., 144, 215--235, 2003.

\bibitem{Poppenberg} M. Poppenberg, {\em On the local wellposedness of quasilinear Schr\"odinger equations in arbitrary space dimension}, J. Diff. Eqs., 172, 83--115, 2001.

\bibitem{Sanz-Serna:84} J.M. Sanz-Serna, {\em Methods for the numerical solution of the nonlinear Schr\"odinger equation}, Math. Comp., 43, 21--27, 1984.

\bibitem{ShenWang:13} J. Shen and Z.Q. Wang, {\em Error analysis of the Strang time-splitting Laguerre-Hermite/Hermite collocation methods for the Gross-Pitaevskii equation}, Found. Comput. Math., 13(1), 99--137, 2013.

\bibitem{Strang:68} G. Strang, {\em On the construction and comparison of difference schemes}, SIAM J. Numer. Anal., 5, 506--517, 1968.

\bibitem{st:1993} M. Suzuki, {\em Improved Trotter-like formula}, Phys. Lett. A, 180, 232--234, 1993.

\bibitem{TahaAblowitz:84} T.R. Taha and M.J. Ablowitz, {\em Analytical and numerical aspects of certain nonlinear evolution equations. II Numerical, nonlinear Schr\"odinger equation}, J. Comput. Phys., 55, 203--230, 1984.

\bibitem{Thalhammer:08} M. Thalhammer, {\em High-order exponential operator splitting methods for time-dependent Schr\"odinger equations}, SIAM J. Numer. Anal.,  46, 2022--2038, 2008.

\bibitem{Thalhammer1} M. Thalhammer, {\em Convergence analysis of high-order time-splitting pseudospectral methods for nonlinear Schr\"odinger equations}, SIAM J. Numer. Anal., 50(6), 3231--3258, 2012.




\bibitem{WeidemanHerbst:86}J.A.C. Weideman and B.M. Herbst,  {\em Split-step methods for the solution of the nonlinear {S}chr\"odinger equation}, {SIAM J. Numer. Anal.},  23, 485--507, 1986.

\bibitem{Yoshida:90} H. Yoshida, {\em Construction of higher order symplectic integrators}, Phys. Lett. A, 150, 262--268, 1990.


\end{thebibliography}

\end{document}